\documentclass[reqno,twoside,11pt]{amsart}

\usepackage{amsmath,amsfonts,calrsfs,fullpage,amssymb,xcolor,verbatim,eucal,yfonts,mathrsfs,graphics,caption}
\usepackage{graphics}
\usepackage{caption} 

\usepackage{mathtools}
\mathtoolsset{showonlyrefs}

 \usepackage{xcolor}
\newtheorem{Theorem}{Theorem}[section]
\newtheorem{Definition}{Definition}
\newtheorem{Proposition}[Theorem]{Proposition}
\newtheorem{Lemma}[Theorem]{Lemma}

\newtheorem{Remark}[Theorem]{Remark}

\newtheorem{Example}[Theorem]{Example}

\makeatletter
\@addtoreset{equation}{section}

\makeatother

\def\a{\alpha}
\def\al{\alpha}

\def\Z{\mathbb Z}

\def\R{\mathbb R}

\def\T{\mathbb T}

\def\E{\mathbb E}
\def\P{\mathbb P}

\def\eps{\varepsilon}
\def\ds{\displaystyle}

\def\vs{\vspace{.1mm}\\}
\def\div{\operatorname{div}}

\newcommand{\vt}{\vartheta}
\newcommand{\Hess}{\operatorname{Hess}}
\newcommand{\supp}{\operatorname{supp}}
\newcommand{\ind}{\mathbb{I}}
\newcommand{\St}{\mathrm{St}}

\newcommand{\dist}{\operatorname{dist}}

\title{Exponential decay of mass for inertial coalescing particles with Hamiltonian noise}\date{}

\author[S. Cerrai]{Sandra Cerrai}
\address{Department of Mathematics\\
university of Maryland\\
}
\email{cerrai@umd.edu}
\thanks{S. Cerrai was partiallly supported by NSF Grant DMS-2348096 (2024-2027), {\em 
 Multiscale Analysis of Infinite-Dimensional Stochastic Systems}}
\author[F. Flandoli]{Franco Flandoli}
\address{Scuola Normale Superiore\\
}
\email{franco.flandoli@sns.it}
\thanks{F. Flandoli was supported by ERC {\em AdG NoisyFluid n.~101053472}}

\author[M. xie]{Mengzi xie}
\address{Institut f\"{u}r Mathematik\\
Technische universit\"{a}t Berlin\\
}
\email{xie@tu-berlin.de}

\subjclass[2010]{}

\keywords{}

\begin{document}

\maketitle

\tableofcontents

\begin{abstract}
We study a system of $N$ inertial particles on a two-dimensional
torus $\T^2$, evolving under a second-order stochastic dynamics with
position-dependent friction $\lambda$ and noise amplitude $\sigma$,
and undergoing coalescence at rate $R_0$ when their distance falls
below a threshold $\delta$. In the joint small-mass / small-correlation
limit $\mu(\eps)\to 0$, $\mu(\eps)/\eps\to\al\in(0,\infty)$, the
empirical measure of the surviving particles converges to a stochastic
continuity equation with inertial drift~$g_\al$. Assuming that $\sigma$
is tangent to the level sets of a Hamiltonian $H=h_1(x_1)\,h_2(x_2)$
satisfying mild non-degeneracy and convexity-type conditions, and that
$\lambda$ and the amplitude $\rho$ of $\sigma$ along $\xi=\nabla^\perp H$
are aligned with $H$, we prove that the expected total mass decays
exponentially in time, with an explicit rate depending on $\al$ and on
the values of $\lambda$ and $\rho$ on the separatrix $\{H=0\}$. The
proof rests on a cell-by-cell analysis of the sign of $\div\,g_\al$ on
the level sets of $H$, showing that the inertial drift pushes
trajectories toward the separatrix at a quantitative rate.
\end{abstract}

\section{Introduction}

\subsection{Inertial particles and the small-mass limit}

The motion of small inertial particles transported by a fluid is a
classical and physically rich problem at the interface of fluid
dynamics, kinetic theory, and stochastic analysis. A particle of mass
$\mu>0$ immersed in a fluid with velocity field $u(t,x)$ is, to a first
approximation, governed by a balance between its inertia and the drag
exerted by the surrounding fluid, leading to a second-order Newton-type
equation in which the small parameter $\mu$ multiplies the
acceleration. When $\mu$ is small, the particle is closely entrained
by the flow, and one expects the dynamics to reduce to a first-order
equation for the position alone. Quantifying this reduction, and
identifying the residual effects of inertia that survive the limit
$\mu\to 0$, is the basic problem of \emph{small-mass asymptotics}.

The picture is enriched, and made considerably more interesting, when
the driving fluid velocity is itself a random process with a small
correlation time $\eps>0$, modelling the fast fluctuations of a
turbulent or otherwise irregular environment. The two small parameters
$\mu$ and $\eps$ then compete, and the limit $(\mu,\eps)\to(0,0)$
depends qualitatively on the ratio $\mu/\eps$.
In the language of fluid mechanics, this ratio is essentially the
\emph{Stokes number}
\[
\St\;:=\;\frac{\tau_p}{\tau_f},
\]
the dimensionless ratio between the particle relaxation time
$\tau_p\sim\mu/\lambda$ (the timescale on which inertia is damped by
drag) and the characteristic correlation time $\tau_f\sim\eps$ of the
ambient fluid. The Stokes number is the basic dimensionless
parameter governing the response of a heavy particle to a fluctuating
flow.

For the purpose of a simplified description, we divide the range of
Stokes numbers in three regimes. When $\mathrm{St}\gg1$, particles are very
inertial, large and heavy compared to the turbulent eddies; particles receive
several small random kicks from the eddies and perform a sort of inertial
Brownian Motion. This is the so-called Abrahamson regime \cite{Abr}, well
understood also from the viewpoint of stochastic modeling of turbulence
\cite{Pap} (the approach further developed here). When $\mathrm{St}\ll1$,
particles behave like traces, they follow the fluid streamlines and the
coalescence rate has been understood by means of the Saffman-Turner theory
\cite{Saf}; also that case is subsceptible of a rigorous analysis by a
stochastic model of turbulence \cite{FlaHuangJSP}, based on the approach to
coagulation developed in \cite{Hamm}, \cite{FlaHuang}. But, when the values of
the Stokes number are intermediate, say $\mathrm{St}=O(1)$, particles and
turbulent eddies are roughly comparable, and the behavior is much more
difficult. Several contributions in the Physics literature are devoted to this
case, see for instance \cite{Bec}, \cite{BecGuMe}, \cite{Dou}, \cite{Fal},
\cite{Meh}, \cite{Pum}, \cite{Wil}. Many heuristic and physically deep ideas have been presented in these works, in some cases concerned with concentration effects produced by turbulence eddies, a fact also at the core of the proof given here. Our work deals rigorously with a concentration property and controls the quantitative consequence of the exponential decay of particle density due to coalescence.

 The
parameter
\begin{equation}
	\label{regime}
	\al\;:=\;\lim_{\eps\to 0}\frac{\mu(\eps)}{\eps}
\end{equation}
that we shall use throughout the paper plays exactly this role of an
effective Stokes number in the small-mass / small-correlation
asymptotics: $\al=0$ corresponds to $\St\to 0$ (passive-tracer limit),
$\al=\infty$ to $\St\to\infty$ (ballistic limit), and finite
$\al\in(0,\infty)$ to the genuinely inertial regime.

When $\mu\ll\eps$ (i.e., $\St\ll 1$), the
particle relaxes to the quasi-static velocity before the fluid has
time to decorrelate, and a Stratonovich-type limit is recovered.
When $\mu\gg\eps$ (i.e., $\St\gg 1$), the fluid decorrelates before
the particle responds, and an It\^o-type limit emerges, with an
additional drift produced by the interaction between inertia and the
noise. The intermediate
regime, where $\mu/\eps\to\al\in(0,\infty)$, interpolates between
these two extremes and produces a  novel correction term, 
the \emph{inertial It\^o drift}, which depends explicitly on $\al$
and vanishes only in the Stratonovich limit $\al=0$.

\subsection{The particle system with  inertial It\^o drift}

We consider here the
second-order system 
\begin{equation}\label{fine1}
\left\{
\begin{array}{l}
\ds{dx_{i, \epsilon}(t)=v_{i, \epsilon}(t)\,dt,\ \ \ \ \ x_{i, \epsilon}(0) \in\,\mathbb{R}^2,}\\[6pt]
\ds{\mu(\epsilon)\,dv_{i, \epsilon}(t)=-\lambda(x_{i, \epsilon}(t))v_{i, \epsilon}(t)\,dt+\sigma(x_{i, \epsilon}(t))z_\eps(t)\,dt,\ \ \ \ \ \ v_{i, \epsilon}(0) \in\,\R^2,}
\end{array}\right.
\end{equation}
for $i=1,\ldots,N$, 
where $z_\epsilon$ is the Ornstein--Uhlenbeck process with correlation
time $\eps>0$
\[\epsilon\,d z_\epsilon(t)=-z_\epsilon(t)\,dt+dw(t),\ \ \ \ \ z_\epsilon(0)=0,\]
for a one-dimensional standard Brownian motion $w(t)$, $t\geq 0$, defined on the stochastic basis $(\Omega, \mathcal{F}, \{\mathcal F_t\}_{t\geq 0}, \P)$. Here $\lambda:\T^2\to(0,+\infty)$ is a $C^2$ friction coefficient bounded below by some $ l_0>0$, and $\sigma:\T^2\to \R^2$ is a $C^2$ noise amplitude, where $\T^2=\T\times \T$ and $\T=\R/L\,\Z$, for some $L>0$.

Following Theorem 2.2. of \cite{CFX} (see also \cite{hmdvw}, \cite{pav-stu} and \cite{PS-2} for related literature), the limiting dynamics of the inertial particles, in the regime \eqref{regime}, is described by the equation 
 \begin{equation}\label{fine2}
dx_i^\al(t)=(\lambda^{-1}\sigma)(x_i^\al(t))\circ dw(t)-g_\al(x_i^\al(t))\,dt,\end{equation}
where the inertial drift is given by
\[g_\alpha(x)=\frac{\alpha}{2(\lambda(x)+\alpha)} \Gamma(x),\ \ \ \ \ \Gamma(x):=D(\lambda^{-1}\sigma)(x)(\lambda^{-1}\sigma)(x)+(\sigma \sigma^t)(x)\frac{\nabla \lambda(x)}{\lambda^3(x)}.\]

Notice that when $\alpha=0$, the drift $g_\alpha(x)$ vanishes, so that the limiting equation reduces to the Stratonovich SDE associated with the first-order
Wong--Zakai approximation
\[dx_i(t)=(\lambda^{-1}\sigma)(x_i(t))\circ dw(t),\ \ \ \ \ \ x^i(0)=x^i_0 \in\,\mathbb{R}^2.\]
When $\alpha=+\infty$,  the limiting equation takes the form
\[dx_i(t)=(\lambda^{-1}\sigma)(x_i(t))dw(t)-\frac 12 (\sigma \sigma^t)(x_i(t))\frac{\nabla \lambda(x_i(t))}{\lambda^3(x_i(t))}\,dt.\]
 The intermediate cases $\al\in(0,\infty)$ produce a
 new drift, $g_\al$, interpolating between the two, that captures the
joint effect of finite inertia and finite correlation time.

In \cite{CFX}, this general result is applied to several physical
situations of interest, including particles in turbulent flows, subject
to Stokes drag, the centrifugal effect, and the so-called
\emph{turbophoretic drift}. In the present paper, we apply it to the analysis of the long-time behavior of   a model of
\emph{coalescing particles}, where particles disappear upon collision,
and the empirical measure of survivors is shown to converge to the
solution of a stochastic transport equation with a quadratic loss
term. 

\subsection{Coalescing particles and the limiting transport equation}

We consider $N$ inertial particles in $\T^2$, each governed by the
limiting equation \eqref{fine2}, and undergoing pairwise coalescence at rate $R_0$ when their
distance falls below a threshold $\delta$: when two particles collide,
one of them is removed. The surviving particles obey equation
\begin{equation}\label{fine3}
dx_i^\al(t)=(\lambda^{-1}\sigma)(x_i^\al(t))\circ dw(t)-g_\al(x_i^\al(t))\,dt,\ \ \ \ \ \ \ \  i\in I_N(t),
\end{equation}
with the same Brownian motion $w(t)$ for all particles (a passive-tracer
type interaction with a single random fluid), where $I_N(t)$ is the
set of surviving indices at time $t$. As established in \cite{FlaHuang}, the empirical measure
\[
\mu_t^N(dx)=\frac{1}{N}\sum_{i\in I_N(t)}\delta_{x_i^\al(t)}(dx)
\]
converges weakly in probability, as $N\to\infty$, to a measure
$f_\al(t,x)\,dx$, where $f_\al$ is the unique bounded weak solution of
the stochastic transport equation
\begin{equation}\label{fine33}
df_\al+\div\!\big(f_\al\,(\lambda^{-1}\sigma)\big)\circ dw_t-\div(f_\al\,g_\al)\,dt=-R_0\,f_\al^2\,dt,\qquad f_\al|_{t=0}=f_0.
\end{equation}
The first three terms on the left-hand side encode the usual
continuity equation associated with the SDE
\eqref{fine3}, while the right-hand side is the
quadratic loss term produced by coalescence. 

In contrast to the linear
continuity equation, \eqref{fine33} does not preserve mass. Actually,
the integral
\[
M_\al(t):=\int_{\T^2}f_\al(t,x)\,dx,
\]
is non-increasing in $t$, and the rate at which it decays measures
the effectiveness of the coalescence mechanism.
The basic question we address is how fast does $M_\al(t)$ go to
zero. A naive bound based on \eqref{fine33} alone is easily
obtained. Integrating in $x$ and  taking expectations,  one finds
\[
\frac{d}{dt}\E\,M_\al(t)\;\leq\;-\frac{R_0}{|\T^2|}\,\big(\E\,M_\al(t)\big)^2,
\]
which integrates to the algebraic decay
\[
\E\,M_\al(t)\;\lesssim\;\frac{|\T^2|}{R_0\,t},\qquad t\to\infty.
\]
This bound is universal, as  it makes no use of the structure of the
underlying SDE or of any geometric features of the data, but it is
only polynomial. Without further information on the dynamics, nothing
forces particles to come close, and the loss term $-R_0 f_\al^2$ alone
is insufficient to produce a faster rate. The point of the present
work is that, under suitable geometric assumptions on the data, the
inertial drift $g_\al$, emerging from the limiting procedure described above, actively concentrates particles on a
one-dimensional set, the separatrix of an underlying Hamiltonian, where collisions become unavoidable, and the resulting decay of
$M_\al(t)$ is \emph{exponential} in time, with a rate that we compute
explicitly in terms of $\al$ and the friction and noise amplitude on
that set.

\subsection{Setting and main result}

As mentioned above, we work in dimension $d=2$ on the torus $\T^2=\T\times\T$, with
$\T=\R/L\Z$. Here, we assume the diffusion $\sigma:\T^2\to\R^2$ is
\emph{tangent} to the level sets of a Hamiltonian $H:\T^2\to\R$. Namely
\begin{equation}\label{fine4}
\sigma(x)\cdot\nabla H(x)=0,\qquad x\in\T^2.\tag{T}
\end{equation}
In two dimensions, this forces $\sigma=\rho\,\xi$ with
$\xi:=\nabla^\perp H$ and $\rho\in C^2(\T^2)$, so that the noise is a
scalar amplification of the Hamiltonian vector field. The Hamiltonian
itself is taken to be {factorized},
\begin{equation}\label{fine42}
H(x)=H(x_1,x_2)=h_1(x_1)\,h_2(x_2),
\end{equation}
with $h_1,h_2\in C^3(\T)$ satisfying mild non-degeneracy and
convexity-type conditions: each $h_i$ has only simple zeros, only
non-degenerate critical points, and $h_i\,h_i''\leq 0$ on $\T$
(see conditions (H1), (H2) and (H3) in the body of the paper). These conditions
are satisfied, for instance, by the Greengard-Thomann fields (see \cite{MB})
\[H(x)=\sin x_1 \cos x_2,\ \ \ \ \ \ \ \ x_i\in\,\T=\R/2\pi\Z,\]
but, as we will show in detail, they  cover a much wider
class of profiles. The separatrix
$\mathcal S:=\{H=0\}=(Z_1\times\T)\cup(\T\times Z_2)$, where
$Z_i=\{h_i=0\}$, decomposes $\T^2$ into a finite union of open
{cells}, each containing a unique {center} (joint critical
point of $h_1$ and $h_2$).

We further impose an \emph{alignment} condition: the friction
$\lambda$ and the noise amplitude $\rho$ are constant on level sets
of $H$, i.e., there exist
$l,r \in\,C^2(\R)$, with
$l\geq l_0>0$ and $r^2\geq r_0>0$, such
that
\begin{equation}\label{fine43}
\lambda(x)=l(H(x)),\qquad \rho(x)=r(H(x)),\ \ \ \ \ x \in\,\T^2.\tag{A}
\end{equation}
In Lemma~\ref{fine10} it is proven that under (T), (A), and (H1), (H2) and (H3), the inertial drift $g_\al$ takes the
remarkably clean form
\begin{equation}\label{fine49}
g_\al(x)=\beta_\al(x)\,\nu(x)^2\,D\xi(x)\,\xi(x),
\end{equation}
where
\[\nu(x):=\rho(x)/\lambda(x),\ \ \ \ \ \ \beta_\al(x):=\al/(2(\lambda(x)+\al))=\al/(2(l(H(x))+\al)).\] 
Moreover, in Lemma \ref{fine25} it is proven that
\begin{equation}\label{fine49-bis}
\nabla H(x)\cdot D\xi(x)\,\xi=H(x)\,\Lambda(x),	
\end{equation}
with $\Lambda\geq 0$ under (H3).  
Expressions \eqref{fine49} and \eqref{fine49-bis} bring to light the geometric significance of $g_\al$. Geometrically, $D\xi\,\xi$ is the acceleration along the Hamiltonian
flow $\dot x=\xi(x)$, and  identity  \eqref{fine49-bis} shows that $g_\al$ has a non-negative
component in the direction of $\nabla H$ on $\{H>0\}$ and a non-positive
one on $\{H<0\}$. Equivalently, the deterministic drift $-g_\al$ in the
limiting SDE points toward the separatrix $\{H=0\}$ on both sides, away
from the centers where $|H|$ is maximal.

\smallskip
The main result of this  paper consists in showing that if the Hamiltonian H satisfies hypotheses (H1), (H2) and (H3) and  conditions \eqref{fine4} and  \eqref{fine43} hold, then, if we also assume that the   initial density $f_0$ is
supported away from the centers $M_1\times M_2$ of $H$,   for every
$\al>0$ there exists a constant $c_\al$ such that 
\begin{equation}\label{fine5}
\E\!\int_{\T^2}\!f_\al(t,x)\,dx
\;\leq\;\frac{c_\al}{R_0}\,\exp\!\left(-\frac{\al\,r^2(0)}{l^2(0)(l(0)+\al)}\,\bar\kappa\,t\right),\qquad t\geq 0,
\end{equation}
for some constant $\bar\kappa>0$ independent of $\al$ and $R_0$.

The decay rate has a transparent dependence on the parameters. As
$\al\to 0^+$, it vanishes proportionally to $\al$, consistently with
the fact that $g_\al\to 0$ in the Stratonovich limit, and there is no
inertial mechanism to force coalescence. As $\al\to\infty$, the rate
saturates at 
\[r^2(0)/l^2(0)\,\bar\kappa,\]
reflecting the fact that the inertial-It\^o drift is bounded uniformly
in $\al$ once $\al$ is large. The rate depends only on the values of
$l$ and $r$ {on the separatrix}: large
$r(0)$ enhances the decay (stronger noise drives particles
across the separatrix more vigorously), while large
$l(0)$ suppresses it (stronger friction damps the
noise-driven motion). The behavior of $l$ and
$r$ in the cell interiors enters only through the prefactor
$c_\al$ and the constant $\bar\kappa$.

It is worth noting that the assumption that $\supp f_0$ avoids the
centers cannot be removed. Actually, a particle initially placed exactly at a
center is a fixed point of the limiting flow $\dot x=\xi(x)$, and the
inertial drift vanishes there. The condition rules out this
non-generic configuration.

\subsection{Strategy of the proof}

The proof has two distinct stages. The first one is an estimate
\emph{at the level of the SDE}. Namely,  we show that the characteristic
flow $\phi_\al$ associated with the drift $-g_\al$ and the noise
$\lambda^{-1}\sigma$ contracts the Hamiltonian, in the sense that
\begin{equation}\label{intro2}
|H(\phi_\al(t,x))|\;\leq\;
|H(x)|\,\exp \left(-c_*\,\al\,t\right),\ \ \ \ \ t\geq 0,
\end{equation}
on the part of $\T^2$ away from the centers, with $c_*>0$ uniform.
The second stage transfers this contraction to the SPDE
\eqref{fine33} and combines it with the quadratic loss term
to produce the exponential decay of $M_\al(t)$. We sketch the main
ingredients in turn.

\smallskip
\noindent\emph{Step 1.} Under (T) and (A),
$g_\al$ takes the form \eqref{fine49}, and a direct
computation shows that
\[
\div\,g_\al(x)
\;=\;\frac{\al}{l(H(x))+\al}\,n^2(H(x))\,D_H(x)
\;-\;H(x)\,\Lambda(x)\,\Theta_\al(H(x)),\ \ \ \ x \in\,\T^2,
\]
where $n=r/l$,     $D_H:=-\det\Hess H$, $\Lambda$ is defined by the algebraic
identity \eqref{fine49-bis} proved in Lemma \ref{fine25}, and $\Theta_\al$ is an
explicit function of $H$. The crucial structural feature is that the
second term on the right hand side above  carries an explicit factor $H$ and {vanishes on the
separatrix}, so the sign of $\div\,g_\al$ near $\mathcal S$ is
controlled entirely by $D_H$.

\smallskip
\noindent\emph{Step 2.} The factorized
form of $H$, combined with (H1), (H2) and (H3), gives a complete description
of the geometry of $\mathcal S$ and of the cells. In each cell, we
identify three distinguished finite sets of points,
the \emph{centers} $M_1\times M_2$ (joint critical points of $H$, on
the cell interior), the \emph{corners} $Z_1\times Z_2\subset\mathcal S$
(saddles of $H$), and the \emph{midpoints}
$\mathcal M=(Z_1\times M_2)\cup(M_1\times Z_2)\subset\mathcal S$.
A key technical lemma shows that \[D_H(x)\geq 0,\ \ \ x \in\,\mathcal S,\ \ \ \ \ \ \ \ \{D_H=0\}\cap\mathcal S=\mathcal M,\] and that at each midpoint
$D_H$ vanishes to exact order $2$ in the longitudinal direction with
quantitative two-sided bounds. A second lemma shows that
\[\Lambda(x)\geq 0,\ \ \ \ x \in\,\T^2,\ \ \ \ \ \ \ \ \{\Lambda=0\}=M_1\times M_2,\] so that
$|H|$ is non-increasing along the characteristic flow.

\smallskip
\noindent\emph{Step 3.} 
The two pieces above
combine to a dichotomy: outside a small neighborhood of
$M_1\times M_2$ and small balls of radius $\vt$ around the midpoints,
$\div\,g_\al$ is bounded below by a positive constant times
$\al\vt^2$, while inside those balls $|\div\,g_\al|$ is at most of
order $\al\vt$. Optimizing over $\vt$ produces the contraction rate
$c_*\al$ at the level of the flow.

\smallskip
\noindent\emph{Step 4.} The
sign-control dichotomy is only useful if one can quantify the
{time} spent by the characteristic flow inside the bad balls
around each midpoint. Lifting to a single cell and projecting onto
the coordinate transverse to the separatrix at a midpoint, this reduces to
an occupation-time estimate of the form
\[
\E\Big(\!\int_s^{s+1}\ind_{[p-\vt,p+\vt]}(\phi_{\al,i}(r,x))\,dr\,\Big|\,\mathcal F_s\Big)\;\leq\;\kappa_0\,\vt,
\]
which says that the projected component $\phi_{\al,i}$ spends at most
a fraction $\kappa_0\vt$ of any unit time interval inside an interval
of width $2\vt$ around the relevant coordinate of the midpoint. The
proof is based  on {Tanaka's formula} applied to the continuous
semimartingale $\phi_{\al,i}(\cdot,x)$ at each level $a$. The
resulting expression for the local time $L^a$, combined with the
Lipschitz bound on $r\mapsto(r-a)^+$ and the cell-invariance property
that confines $\phi_{\al,i}$ to an interval of length $<2\pi$, yields
a uniform bound on $\E(L^a_{s+1}-L^a_s|\mathcal F_s)$; the
occupation-time formula then converts this into the desired bound on
the time spent in the projected interval. This is the technical heart
of the quantitative side of the argument. Namely,  it is what turns the
qualitative sign dichotomy into the explicit rate $c_*\al$.

\smallskip
\noindent\emph{Step 5.} The contraction \eqref{intro2} is then translated
to the density $f_\al$ via the characteristics representation. The
key observation is that the support of $f_\al$ is, on average, pushed
towards $\mathcal S$ at exponential rate, which forces the
$L^2$-norm of $f_\al$ along the characteristics to remain bounded
below by the $L^1$-norm raised to a power, and the loss term
$-R_0 f_\al^2$ then drives $M_\al(t)$ to zero exponentially. This
last step is carried out via a Gronwall-type estimate combined with a
careful localization argument near $\mathcal S$.

\subsection{Organization of the paper}

Section~\ref{sec3}  sets up
the precise framework, states all assumptions, and
recall the basic geometric facts. Proposition~\ref{fine17}
characterizes the structure of $h_i$ under (H1), (H2) and (H3), and
Proposition~\ref{fine10} establishes the reduction
\eqref{fine49} for $g_\al$. The main result,
Theorem~\ref{thm:main}, is stated at the end of that section.

Section~\ref{sec4} establishes the structural decomposition
of $\div\,g_\al$ into the bulk and boundary contributions, and proves
the key vanishing properties of $D_H$ on $\mathcal S$.
Section~\ref{fine64} proves the monotonicity of $|H|$ along
the characteristic flow and  carries out the sign-control
argument that yields the lower bound on $\div\,g_\al$ outside a thin
neighborhood of $\mathcal M$, and the corresponding upper bound on
the bad set.
Section~\ref{fine56} translates these estimates into the
contraction estimate at the level of the SDE. Section \ref{fine26} studies two aspects of the flow on a cell, the It\^o form and the non-degeneracy near midpoints.  
Section~\ref{sec8} transfers the contraction to the density
$f_\al$ and concludes the proof of Theorem~\ref{thm:main}.
Section~\ref{sec9}  discusses
possible extensions of the main result, in particular, beyond the
factorized form of $H$ and beyond the alignment condition (A), and discusses examples of $h_1,h_2$

\section{Notations,  assumptions and main result}
We introduce two functions  $h_1$ and $h_2$ in $C^3(\T)$ and define the Hamiltonian
\begin{equation}\label{fine6}H(x_1,x_2)=h_1(x_1)\,h_2(x_2),\ \ \ \ \ \ (x_1, x_2) \in\,\T^2.
\end{equation}
Throughout the paper, we shall impose that the functions $h_i$, for $i=1, 2$, satisfy the following conditions.
\begin{itemize}
\item[(H1)]  The function $h_i$ has only simple zeros. Namely, \[h_i(a)=0\Longrightarrow h_i'(a)\neq 0.\]

\item[(H2)] The function $h_i$ has only Morse critical points. Namely, 
\[h_i^\prime(b)=0\Longrightarrow h_i^{\prime \prime}(b)\neq 0.\]

\item[(H3)] The following convexity-type condition holds, 
\[h_i(x)\,h_i''(x)\leq 0,\ \ \ \ \ x\in\T.\]

\end{itemize}

Without loss of generality, we further assume the normalization
$\max_\T|h_i|=1$, for $i=1,2$. With this normalization, the range of $H$ is contained
in $[-1,1]$. In Section \ref{sec3}, we will discuss the properties of the Hamiltonian $H$, and we will provide a wide class of functions $h$ fulfilling conditions (H1), (H2) and (H3). In the meantime, we  introduce the following two sets associated with the functions $h_i$
\[
Z_i:=\{a\in\T: h_i(a)=0\},\qquad M_i:=\{b\in\T: h_i'(b)=0\}.
\]

Now, we introduce two structural conditions on $\sigma$ and $\lambda$. The first one is a {\em tangency} condition that ensures that the noise conserves $H$.
\begin{itemize}
\item[(T)] It holds
\[\sigma(x)\cdot\nabla H(x)=0,\ \ \ \ \ \ x\in\T^2.\]
\end{itemize}

In dimension~$2$, condition (T) means that $\sigma(x)$ is collinear with
$\xi(x):=\nabla^\perp H(x)$ at every point. Hence, there exists a $C^2$ scalar function $\rho:\T^2\to\R$ such that
\begin{equation}\label{fine7}
\sigma(x)=\rho(x)\,\xi(x),\qquad x\in\T^2.
\end{equation}

The second condition
ensures that $\lambda$ and $\rho$ are constant along level sets of $H$.
\begin{itemize}
\item[(A)]  There exist
$l,r\in C^2([-1,1])$, with $l(x)\geq l_0>0$ and $r^2(x)\geq r_0>0$, such that
\[
\lambda(x)=l(H(x)),\qquad \rho(x)=r(H(x)),\qquad x\in\T^2.
\]
\end{itemize}

To express $g_\al$ explicitly under conditions (T) and (A), we introduce the mappings 
\begin{equation}\label{fine8}n(h):=\frac{r(h)}{l(h)},\ \ \ h \in\,[-1,1],\ \ \ \ \ \nu(x):=\frac{\rho(x)}{\lambda(x)}=n(H(x)),\ \ \ \ \ \ x \in\,\T^2,
\end{equation}
and
\begin{equation}\label{fine9}
\beta_\alpha(x):=\frac{\al}{2(\lambda(x)+\al)}=\frac{\al}{2(l(H(x))+\al)},\ \ \ \ \ \ x \in\,\T^2.
\end{equation}
Notice that, as a consequence of condition (A), the vectors $\nabla \lambda(x)$, $\nabla \nu(x)$ and $\nabla \beta_\alpha(x)$ are all parallel to $\nabla H(x)$, for all $x \in\,\T^2$. 

With  these notations, equation \eqref{fine3} can be rewritten as
\begin{equation}\label{sfl8}
dx_\al(t)=\nu (x_\al(t))\xi(x_\al(t))\circ dw(t)-g_\al(x_\al(t))\,dt.
\end{equation}
Moreover,  equation \eqref{fine33} can be rewritten as
\begin{equation}\label{sfm10}
df_\alpha(x,t)  +\operatorname{div}(f_\alpha(x,t)
\nu(x)\,\xi(x))  \circ dw(t)-\operatorname{div}(
f_\alpha(  x,t)  g_\alpha(  x)  )  dt=-R_{0}%
f_\alpha^{2}(  x,t)  dt,
\end{equation}
and since $\operatorname{div}\xi(x)=0$ and $\nabla\nu(x)\cdot \xi(x)=0$,  we get 
\begin{equation}\label{sfn11}
\begin{array}{l}
\ds{df_\alpha(t,x)=-\nu(x)\xi(x)\cdot \nabla f_\alpha(t,x)\circ dB(t)}\\[10pt]
\ds{\quad \quad \quad \quad \quad \quad \quad +g_{\alpha}(x)\cdot \nabla f_\alpha(t,x)\,dt+\text{div}\,g_{\alpha}(x)\,f_\alpha(t,x)\,dt-R_{0}f_\alpha^{2}(t,x)\,dt.}	
\end{array}
	\end{equation}

According to the structural assumptions we have given, the inertial drift 
\begin{equation}\label{sfm56}g_\alpha(x)=\frac{\alpha}{2(\lambda(x)+\alpha)} \Gamma(x),\ \ \ \ \ \Gamma(x):=D(\lambda^{-1}\sigma)(x)(\lambda^{-1}\sigma)(x)+(\sigma \sigma^t)(x)\frac{\nabla \lambda(x)}{\lambda^3(x)},\end{equation}
has the following simplified expression.

\begin{Lemma}\label{fine10}
Under (T) and (A), with $\nu$ and $\beta_\alpha$ defined as in  \eqref{fine8} and 
\eqref{fine9}, respectively,  we have
\begin{equation}\label{fine11}
g_\al(x)=\beta_\alpha(x)\,\nu^2(x)\,D\xi(x)\xi(x),\ \ \ \ \ \ x \in\,\T^2.
\end{equation}
\end{Lemma}

\begin{proof}
We substitute $\sigma=\rho\,\xi$ and $\lambda^{-1}\sigma=\nu\,\xi$ into
\eqref{sfm56}.
For the first term of $\Gamma(x)$, since $\nabla \nu(x)$ is parallet to $\nabla H(x)$, we get
\[
D(\nu\xi)(x)\,(\nu\xi)(x)=\nu(x)\,(\xi(x)\cdot\nabla\nu(x))\,\xi(x)+\nu^2(x)\,D\xi(x)\,\xi(x)=\nu^2(x)\,D\xi(x)\,\xi(x).
\]
For the second term of $\Gamma(x)$, we have \[(\sigma\sigma^\top)(x)\,v=\rho^2\,(\xi(x)\cdot v)\,\xi(x),\ \ \ \ \ v\in\R^2.\] Hence, if we take $v=\nabla\lambda(x)/\lambda^3(x)$, as $\nabla \lambda(x)$ is parallel to $\nabla H(x)$, we get
\[
(\sigma\sigma^\top)(x)\,\frac{\nabla\lambda(x)}{\lambda^3(x)}
=\frac{\rho^2(x)}{\lambda^3(x)}\,(\xi(x)\cdot\nabla\lambda(x))\,\xi(x)=0.
\]
This gives  $\Gamma(x)=\nu(x)^2\,D\xi(x)\,\xi(x)$ and   \eqref{fine11} follows.
\end{proof}

The present paper is devoted to the study of the long-time behavior of the space average of $f_\alpha(t,x)$ in $\mathbb{T}^2$, for every fixed $\alpha>0$. Our main result is stated here.

\begin{Theorem}\label{thm:main}
Let $H$ be the Hamiltonian defined in \eqref{fine6}, for some functions $h_1, h_2$  in $C^3(\T)$ satisfying
hypotheses (H1), (H2) and (H3). Assume that the diffusion $\sigma$ satisfies the tangentiality condition (T). Moreover, let the friction $\lambda:\T^2\to[ l_0,\infty)$
and the noise amplitude $\rho:\T^2\to(-\infty,\infty)$ be $C^2$ scalar
functions satisfying the alignment condition (A). 

Under these conditions, if $\supp f_0$ is disjoint from $M_1\times M_2$, the set of
the centers of $H$, for every
$\al>0$ there exists $c_\al>0$ such that
\begin{equation}\label{fin1}
\E\int_{\T^2}f_\al(t,x)\,dx\;\leq\;\frac{c_\al}{R_0}\,\exp\!\left(\!-\frac{\al\,r^2(0)}{l^2(0)(l(0)+\al)}\,\bar\kappa\,t\right),\ \ \ \ \ \ \ \  t\geq 0,
\end{equation}
for some constant  $\bar\kappa>0$   independent of  $\al$ or $R_0$.
\end{Theorem}

\begin{Remark}{\em 
The function
\[\al\mapsto \frac{\al\,r^2(0)}{l^2(0)(l(0)+\al)},\]
 vanishes, as $\al\to 0^+$, and saturates at $r^2(0)/l^2(0)$, as $\al\to\infty$. Moreover, it  is monotonically increasing with the same shape as $\al/(l_0+\al)$, up to a multiplicative factor.

For small $\al>0$,
\[
\frac{\al\,r^2(0)}{l^2(0)(l(0)+\al)}\;\sim\;\frac{\al\,r^2(0)}{l^3(0)}.
\]
 Thus, the decay
rate is enhanced by a factor $r^2(0)/l^3(0)$, at
small~$\al$. The dependence on the separatrix values can be interpreted as follows.
\begin{itemize}
\item[-] Large $r(0)$:  strong noise amplitude on the separatrix 
\emph{enhances} the decay, because the noise is what drives the
particle towards the cell boundaries where coalescence happens.
\item[-] Large $l(0)$:  strong friction on the separatrix 
\emph{suppresses} the decay, because friction damps the noise-driven
motion.
\end{itemize}

\smallskip

For large $\al$, the rate saturates at
$r^2(0)/l^2(0)\cdot\bar\kappa$.

\smallskip

 Notice that only the
values at $0$ of the profiles $l$ and $r$
enter the rate. The behavior of $l$ and
$r$ in the cell interiors is absorbed into the constants
$\bar\kappa$ and $c_\al$. }\end{Remark}

\section{Geometric characterization of the Hamiltonian and a few examples}
\label{sec3}

We have the following geometric characterization.

\begin{Proposition}\label{fine17}
A function $h\in C^3(\T)$ satisfies conditions (H1), (H2) and (H3) if and only if there
exists an integer $N\geq 1$ such that
\begin{itemize}
\item[1.] $h$ has exactly $2N$ zeros $a_1<a_2<\cdots<a_{2N}$ in $\T$,
all simple;
\item[2.] $h$ has exactly $2N$ critical points $b_1,\ldots,b_{2N}$,
interlaced with the zeros of $h$,  one $b_j$ in each open interval
$(a_j,a_{j+1})$. Moreover,  at each $b_j$, we have $h(b_i)h''(b_j)< 0$;
\item[3.] on each interval $(a_j,a_{j+1})$, $h$ has constant sign and $|h|$
is concave.
\end{itemize}

\end{Proposition}

\begin{proof}
It is immediate to check that if conditions 1, 2 and 3 are satisfied, then (H1), (H2) and (H3) hold true.

Now, assume that (H1), (H2) and (H3) hold. In view of  (H1), the zeros of $h$ are simple, hence
isolated. Being $\T$  compact, they are finite. Since $h$ is continuous on $\T$ and changes sign at
each simple zero, the number of zeros must be even, say $2N$.
Moreover, $N\geq 1$. Actually, if $h>0$ on all of $\T$, (H3) gives
$h''\leq 0$ everywhere, so 
\[\int_\T h''(x)\,dx\leq 0.\] But by periodicity this implies
\[\int_\T h''(x)\,dx=h'(2\pi)-h'(0)=0,\] forcing $h''\equiv 0$. Then
$h$ is linear and periodic, hence constant. But then every point is a
critical point, with $h''\equiv 0$, violating (H2). The case $h<0$
on $\T$ is symmetric. So $N\geq 1$.

Property 3. is essentially (H3) restated on each interval.
On the interval $(a_j,a_{j+1})$ where $h$ has constant sign condition  (H3) gives $h''(x)\leq 0$, if $h>0$,  and $h''(x)\geq 0$, if $h<0$. So $|h|$ is concave on the interval.
By concavity, $|h|$ has a unique maximum on $[a_j,a_{j+1}]$. Since
$h(a_j)=h(a_{j+1})=0$, this maximum is interior, at some
$b_j\in(a_j,a_{j+1})$, where $h'(b_j)=0$. By (H2), $h''(b_j)\neq 0$ and has the opposite sign
of $h(b_j)$, consistent with the strict concavity at the peak.
Concavity of $|h|$ also forces uniqueness. 
\end{proof}

\medskip

A useful way to construct examples is via Sturm--Liouville-type ODEs.
\begin{Proposition}\label{prop3.3}
Let $W\in C^1(\T)$, with $W>0$ on $\T$. Suppose $h\in C^3(\T)$ is a
non-trivial real-valued periodic solution of
\begin{equation}\label{fine19}h''(x)+W(x)\,h(x)=0,\ \ \ \ \ x\in\T.
\end{equation}
Then $h$ satisfies conditions (H1), (H2) and (H3). 

On the other hand, assume that $h\in C^3(\T)$ is non-trivial and satisfy conditions (H1), (H2) and (H3). If we define
$W:\T\to\R$ by
\begin{equation}\label{fine20}W(x):=\begin{cases}\,-h''(x)/h(x)&\text{if }h(x)\neq 0,\\[2pt]\,-h'''(x)/h'(x)&\text{if }h(x)=0,\end{cases}
\end{equation}
then $W\in C^1(\T)$, $W\geq 0$ on $\T$, and $h$ is a solution of equation \eqref{fine19}.
\end{Proposition}

\begin{proof}
Equation  \eqref{fine19}  is a second-order linear ODE with
$C^1$ coefficient $W$.  The
zero function is the unique solution with initial data $(0,0)$. Hence if
$h(a)=h'(a)=0$ we would have $h\equiv0$,  contradicting the assumption that $h$ is non-trivial.
Therefore, all zeros of $h$ are simple. 

Next, at any $b$ with $h'(b)=0$, we would have $h(b)\neq 0$, and the ODE
gives $h''(b)=-W(b)h(b)$. Since $W(b)>0$ and $h(b)\neq 0$, we obtain $h''(b)\neq 0$.
Finally, we have $h(x)h''(x)=-W(x)h^2(x)\leq 0$, since $W(x)\geq 0$ and this implies (H3).

We omit the proof of the converse and we leave it  to the reader.
 \end{proof}

\subsection*{Concrete examples}

In what follows, we list  several families of functions $h$ satisfying conditions (H1), (H2), and (H3).

\begin{Example}\label{fine21}
{\em For any $k\in\Z\setminus\{0\}$ and $\varphi\in\R$,  the Greengard-Thomann function 
\[
h(x)=\sin(kx+\varphi),\ \ \ \ \ x \in\,\T,
\]
satisfies (H1), (H2) and (H3), with wave number $N=|k|$. Indeed, $h''(x)=-k^2 h(x)$, so
$h''(x)+k^2h(x)=0$. This is the Sturm-Liouville form, with constant $W(x)= k^2>0$. The zeros  of $h$ are at
$kx+\varphi\in\pi\Z$, with $h'(x)=k\cos (k x+\varphi)\neq 0$ there. The
critical points are at $kx+\varphi\in\pi/2+\pi\Z$, with $h''(x)=-k^2 h(x)\neq 0$. Finally, 
$h(x) h''(x)=-k^2h^2(x)\leq 0$.  See Figure \ref{fig:example3.3}, for the level curves of $H(x)$, where $h_{1}(x_{1})=\sin x_{1}$ and $h_{2}(x_{2})=\sin(x_{2}+\pi/2)=\cos(x_{2})$. }      
\end{Example}

\begin{Example}\label{fine22}
{\em The function
\[
h(x)=\big(1+\tfrac14\cos x\big)\sin x
\]
satisfies (H1), (H2) and (H3), with wave number $N=1$. Indeed, $1+\tfrac14\cos x\in[\tfrac34,\tfrac54]$ is positive on $\T$, so $h(x)$ has the same sign as $\sin x$, with simple zeros at $0$ and $\pi$. One computes $h''(x)=-\sin x(1+\cos x)$, so \[h(x)h''(x)=-\sin^2 x(1+\tfrac14\cos x)(1+\cos x)\leq 0,\ \ \ \ \ x \in\,\T,\] with equality only where $h$ vanishes. The critical points solve $h'=\cos x+\tfrac14\cos 2x=0$, which has exactly two solutions in $\T$ (one in the interval $(0,\pi)$, and the other  one in the interval $(\pi,2\pi)$), at each of which $h''\neq 0$ by direct check.
Notice that this example is genuinely beyond pure trigonometric.     See Figure \ref{fig:example3.4}, for the level curves of $H(x)$, where $h_{1}(x_{1})=(1+\frac{1}{4}\cos x_{1})\sin x_{1}$ and $h_{2}(x_{2})=\cos x_{2}$. }             
\end{Example}


\begin{figure}[htbp]
	\centering
	
	\begin{minipage}{0.49\textwidth}
		\centering
		\includegraphics[width=\linewidth]{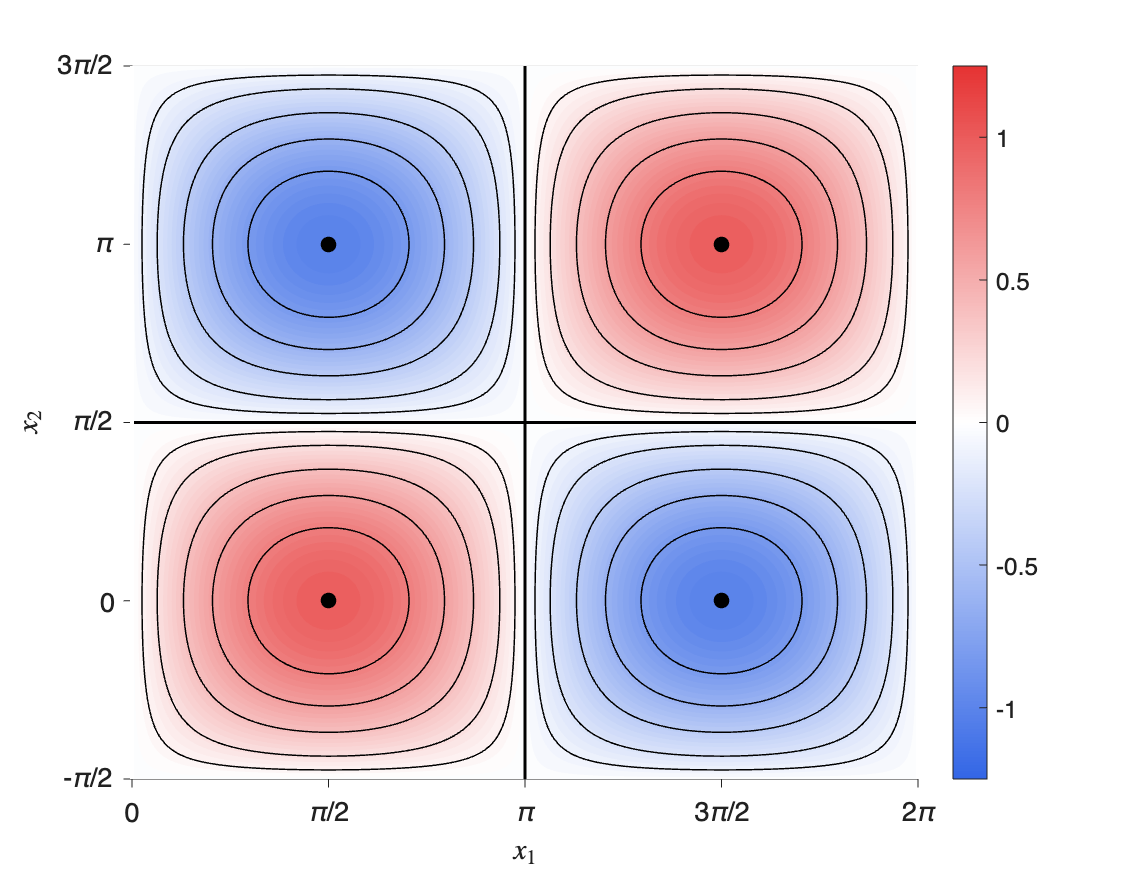}
		\captionsetup{justification=centering}
		\caption{\ \\[10pt]$H(x)=\sin x_1 \cos x_2$}
		\label{fig:example3.3}
	\end{minipage}
	\hfill
	\begin{minipage}{0.49\textwidth}
		\centering
		\includegraphics[width=\linewidth]{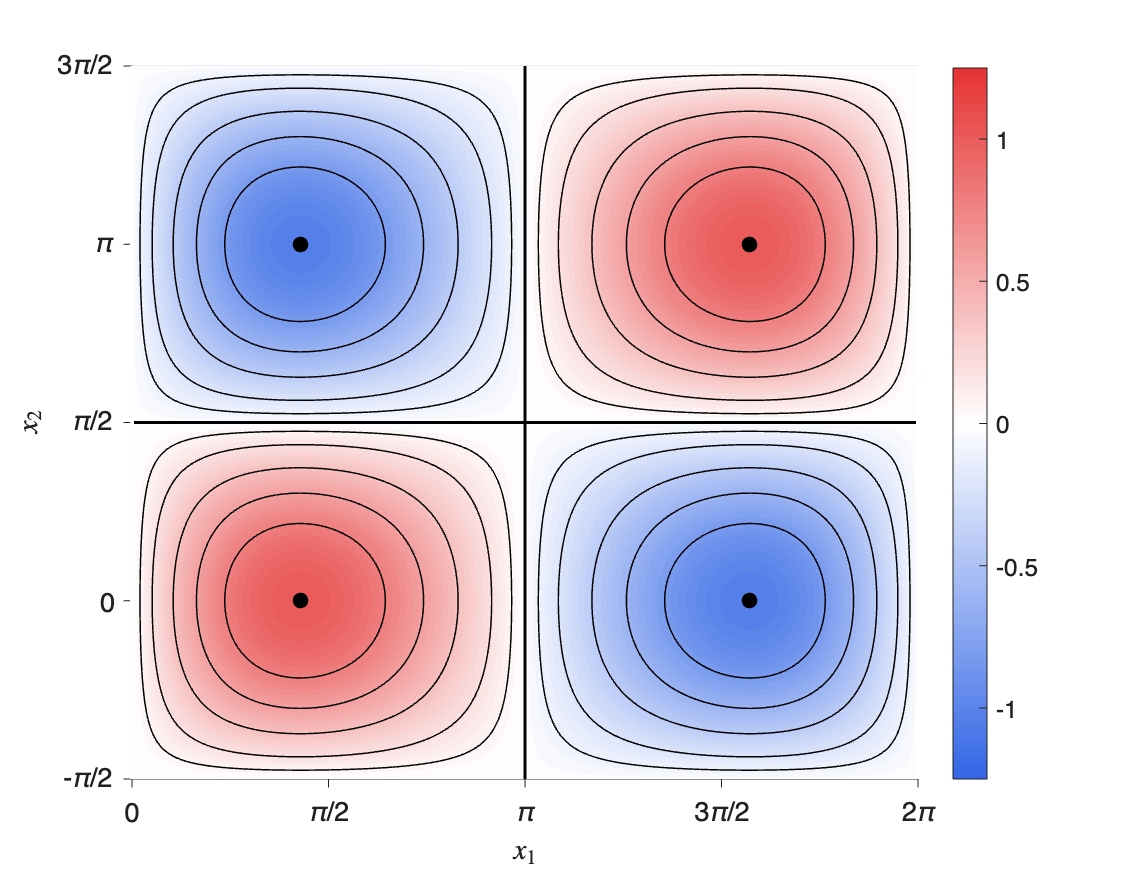}
		\captionsetup{justification=centering}
		\caption{\ \\[10pt]$H(x)=\bigl(1+\tfrac{1}{4}\cos x_{1}\bigr)\sin x_{1}\cos x_{2}$}
		\label{fig:example3.4}
	\end{minipage}
	
\end{figure}


\begin{Example}\label{fine23}
{\em Let $V\in C^1(\T)$ be arbitrary. The eigenvalue problem
\[
-h''(x)+V(x)h(x)=\lambda\,h(x)\quad\text{on }\T
\]
has a discrete set of eigenvalues $ l_0\leq\lambda_1\leq\lambda_2\leq\cdots\to\infty$ with eigenfunctions in $C^3(\T)$.
Thus, if for every $n\geq 0$ such that $\lambda_n>\max_\T V$, we take the   eigenfunction
$h_n$ at $\lambda_n$, we have that it satisfies (H1), (H2) and (H3) thanks to 
Proposition ~\ref{prop3.3} applied to $W:=\lambda_n-V>0$.
Since $\lambda_n\to\infty$ while $V$ is bounded, this gives infinitely
many examples for every $V$. For any non-constant $V$, these are
 non-trigonometric.}
\end{Example}

\subsection{Geometry of $H$ under (H1), (H2) and (H3).}
We recall that we have defined
\[
Z_i=\{a\in\T: h_i(a)=0\},\qquad M_i=\{b\in\T: h_i'(b)=0\}.
\]
Each of set $Z_i$ and $M_i$ is finite and $Z_i\cap M_i=\emptyset$.
Indeed, by (H1), at every $a\in Z_i$ one has $h_i'(a)\neq 0$, so $h_i$
is strictly monotone in a neighborhood of $a$ and $a$ is an isolated
point of $Z_i$. Since $\T$ is compact, an isolated-point set is finite,
so $Z_i$ is finite. By (H2), the same argument applied to $h_i'$ at
every $b\in M_i$ gives that $M_i$ is finite. Finally, if $a\in Z_i\cap M_i$ then $h_i(a)=h_i'(a)=0$, contradicting (H1); hence
$Z_i\cap M_i=\emptyset$.

The critical points of $H$ on $\T^2$ are of two types, centers and corners.
\begin{Definition}
The \emph{centers} of the Hamiltonian $H$are all points  $p=(p_1,p_2)\in M_1\times M_2$.	
\end{Definition}
If $p$ is a center, then
$h_1'(p_1)=h_2'(p_2)=0$, so the off-diagonal entries of
$\text{Hess} H(p)$ vanish, and the diagonal entries are $h_1''(p_1)h_2(p_2)$
and $h_1(p_1)h_2''(p_2)$. By Proposition ~\ref{fine17}, $h_i(p_i)h_i''(p_i)<0$, i.e.,
$\operatorname{sgn}(h_i''(p_i))=-\operatorname{sgn}(h_i(p_i))$. The two
diagonal entries therefore have signs
\[
\operatorname{sgn}(h_1''(p_1)h_2(p_2))=-\operatorname{sgn}(h_1(p_1))\operatorname{sgn}(h_2(p_2))=-\operatorname{sgn}(H(p)),
\]
and similarly for the second diagonal entry. So $\text{Hess}\, H(p)$ is
sign-definite (negative-definite if $H(p)>0$, positive-definite if
$H(p)<0$), and $p$ is a non-degenerate local extremum of $H$.

\begin{Definition}
The \emph{corners} of the Hamiltonian $H$ are the elements of the set
$Z_1\times Z_2$.
\end{Definition}
In this case, due to (H1) we have
\[
\det\text{Hess} H(q)=-(h_1'(q_1))^2(h_2'(q_2))^2<0
\]
so $q$ is a  saddle of $H$.
\begin{Definition}
The \emph{separatrix} of the Hamiltonian $H$ is the set $\mathcal S:=\{H=0\}=(Z_1\times\T)\cup(\T\times Z_2)$.	
\end{Definition}
It is easy to check that $\mathcal S$ is  a finite union of vertical and horizontal lines meeting at the
corners. The complement $\T^2\setminus\mathcal S$ is a finite disjoint
union of open rectangular cells $\mathcal C_{ij}$ on which $H$ has
constant sign.

Each cell contains exactly one center. Actually, if we denote by 
 $\mathcal C$ the  cell with lifted rectangle
$[a_1,b_1]\times[a_2,b_2]$, where $a_1,b_1$ are consecutive zeros of
$h_1$ and $a_2,b_2$ are consecutive zeros of $h_2$, we have
\[
(M_1\times M_2)\cap\overline{\mathcal C}=(M_1\cap[a_1,b_1])\times(M_2\cap[a_2,b_2])=(M_1\cap(a_1,b_1))\times(M_2\cap(a_2,b_2))
\]
where the second equality uses $M_i\cap Z_i=\emptyset$. Thus, since that each
factor has exactly one element by Proposition~\ref{fine17}.2, we conclude that $\mathcal C$ contains exactly one center.

\begin{Definition}
The \emph{midpoints} of the separatrix $\mathcal{S}$ are the elements of the set
$\mathcal M:=(Z_1\times M_2)\cup(M_1\times Z_2).$
\end{Definition}
This is a finite set, disjoint from the corners $Z_1\times Z_2$ and from
the centers $M_1\times M_2$. Moreover, the set $\overline{\mathcal C}$ contains exactly four midpoints
of $\mathcal M$: two of type $Z_1\times M_2$ and two of type $M_1\times Z_2$.

\section{Key algebraic identities and inequalities}\label{sec4}
We start from the following fundamental equality.
\begin{Lemma}\label{fine25}
For every $x \in\,\T^2$ it holds\begin{equation}\label{fine34}
\nabla H(x)\cdot D\xi(x)\,\xi(x)\;=\;H(x)\,\Lambda(x),
\end{equation}
where
\begin{equation}\label{fine44}
\Lambda(x):=2(h_1')^2(x_1)(h_2')^2(x_2)-(h_1')^2(x_1) (h_2 h_2'')(x_2)-(h_1 h_1'')(x_1)(h_2')^2(x_2).
\end{equation}
Moreover
\begin{equation}\label{fine45}
\text{{\em div}}(D\xi\,\xi)(x)\;=\;2\,D_H(x),\end{equation}where
\[D_H(x):=(h_1')(x_1)^2(h_2')^2(x_2)-(h_1 h_1'')(x_1)\,(h_2 h_2'')(x_2).
\]
\end{Lemma}

\begin{proof}
By direct computation, we get
\[\begin{array}{l}
\ds{
D\xi\,\xi=\big(h_1h_1'\,(h_2')^2-h_1h_1'\,h_2h_2'',\;(h_1')^2\,h_2h_2'-h_1h_1''\,h_2h_2'\big).}
	\end{array}
\]
Thus, the dot product $\nabla H\cdot (D\xi\xi)$ expands to
\[h_1h_2[(h_1')^2(h_2')^2-h_1h_1''(h_2')^2]+h_1h_2[(h_1')^2(h_2')^2-(h_1')^2h_2h_2'']
=H\Lambda.\] 
For \eqref{fine45}, the first component of $D\xi\xi$ is a
function of $x_1$ times a function of $x_2$, and similarly for the
second. Thus,  differentiating and adding yields
\[\text{div}(D\xi\xi)=2\big((h_1')^2(h_2')^2-h_1h_1''h_2h_2''\big)=2\,D_H.\]
\end{proof}

\begin{Remark}\label{fine50}
{\em According to hypothesis (H3), we have $\Lambda\geq 0$ on $\T^2$. The function $D_H$ admits the geometric interpretation
\[D_H(x)=-\det\big(\text{Hess}\, H(x)\big).\] Moreover, it is immediate to check that
\begin{equation}
\label{sfm1}
2D_H(x)=\Lambda(x)+[(h_1')^2-h_1 h_1''](x_1)(h_2 h_2'')(x_2)+[(h_2')^2-h_2 h_2''](x_2)(h_1 h_1'')(x_1).	
\end{equation}
}
\end{Remark}

Next, we investigate the structure of $\text{div}\,g_\alpha$. 
\begin{Lemma}\label{fine51}Under (T) and (A), we have
\begin{equation}\label{fine52}
\text{{\em div}}\,g_\al(x)\;=\;\frac{\al}{l(H(x))+\al}\,n^2(H(x))\,D_H(x)
\;-\;H(x)\,\Lambda(x)\,\Theta_\alpha(H(x)),
\end{equation}
where the scalar profile $\Theta_\alpha\in C^1([-1,1])$ is given by
\begin{equation}\label{fine57}
\Theta_\alpha(h):=\frac{\al\,l'(h)\,n^2(h)}{2\,(l(h)+\al)^2}\;-\;\frac{\al\,n(h)\,n'(h)}{l(h)+\al}.
\end{equation}
\end{Lemma}

\begin{proof}
In view of Lemma~\ref{fine10},  under  conditions (T) and (A) we have $g_\al=\beta_\alpha \,\nu^2\,D\xi\xi$. Hence, if we compute the divergence of both sides, we get 
\[
\text{div} g_\al=\nabla(\beta_\alpha\nu^2)\cdot D\xi\xi+\beta_\alpha\nu^2\,\text{div}(D\xi\xi).
\]
Thanks to \eqref{fine45}, the second term above gives \[2\beta_\alpha(x)\nu^2(x)\,D_H(x)=\frac{\al}{l(H(x))+\al}\,n^2(H(x))\,D_H(x),\] which corresponds to the first term on the right hand side in \eqref{fine52}.

For the first term, expand $\nabla(\beta_\alpha\nu^2)=\nu^2\,\nabla\beta_\alpha+2\beta_\alpha\nu\,\nabla\nu$. Under (A),
\[
\nabla\beta_\alpha=-\frac{\al\,l'(H)}{2\,(l(H)+\al)^2}\,\nabla H,\qquad
\nabla\nu=n'(H)\,\nabla H,
\]
both parallel to $\nabla H$. By using \eqref{fine34}, this yields
\[
\nabla\beta_\alpha\cdot D\xi\xi=-\frac{\al\,l'(H)}{2(l(H)+\al)^2}\,H\Lambda,\quad
\nabla\nu\cdot D\xi\xi=n'(H)\,H\Lambda.
\]
Therefore
\[
\nabla(\beta_\alpha\nu^2)\cdot D\xi\xi=H\Lambda\bigg(-\frac{\al\,l'(H)\,n^2(H)}{2(l(H)+\al)^2}+2\beta_\alpha(H)n(H)n'(H)\bigg)=-H\Lambda\,\Theta_\alpha(H),
\]
recalling how $\beta_\alpha$ and $\Theta_\alpha$ are  defined. Combining this with the second term,  \eqref{fine52} follows.
\end{proof}

\begin{Remark}\label{fine46}
{\em  According to Lemma \ref{fine51}, we have
\begin{equation}
\label{sfm12}
\sup_{x \in\,\T^2}\vert\div g_\alpha(x)\vert \leq \frac{c\,\alpha}{ l_0+\alpha}.	
\end{equation}
When  $\sigma=\lambda\xi$, we have $m\equiv 1$, so that $n'\equiv 0$. Then  \eqref{fine57} reduces to \[\Theta_\alpha(h)=\frac{\al l'(h)}{2(l(h)+\al)^2},\]
 and formula \eqref{fine52} becomes
\[
\text{div}\, g_\al(x)=\frac{\al}{\lambda(x)+\al}D_H(x)-\frac{\al\,l'(H(x))\,H(x)\,\Lambda(x)}{2(\lambda(x)+\al)^2}.
\]

When $l=r=1$, then $\Theta_\alpha= 0$ and 
\[\text{div}\, g_\al(x)=\frac{\al}{1+\al}\,D_H(x).\]

Note that the term $H(x)\,\Lambda(x)\,\Theta_\alpha (H(x))$ has a factor $H$ and therefore vanishes on the separatrix $\mathcal S=\{H=0\}$. This is the critical structural feature exploited in the sign-control argument.}
\end{Remark}

Now, we describe the sign and vanishing of the mapping   $D_H$.

\begin{Lemma}\label{fine47}The function $D_H$ is strictly negative at every center $p=(p_1,p_2)\in M_1\times M_2$. Moreover, 
$D_H\geq 0$ on the separatrix $\mathcal S$, with $\{D_H=0\}\cap\mathcal S=\mathcal M$.
\end{Lemma}
\begin{proof}
At a center $p=(p_1, p_2)$, we have $h_1'(p_1)=h_2'(p_2)=0$, and this implies
\[(h_ih_i'')(p_i)<0,\ \ \ \ \ \ i=1,2.\]
Moreover,
\[D_H(p)=-(h_1h_1'')(p_1)\,(h_2h_2'')(p_2),\] 
and hence
  $D_H(p)<0$.
Next, on $\{h_1=0\}$ we have $h_1h_1''=0$, so that \[D_H=(h_1')^2(h_2')^2,\] which is
non-negative and strictly positive  off $M_2$, due to fact that the zeros of $h_1$ are all simple. Moreover, it vanishes exactly when $h_2'=0$, i.e.\ at $Z_1\times M_2$. We can proceed symmetrically on
$\{h_2=0\}$.
\end{proof}

Next, we describe how $D_H$ vanishes around the midpoints of the separatrix.

\begin{Lemma}\label{fine48}
At each midpoint $q\in\mathcal M$, the function $D_H$ vanishes at exact order $2$,
and is non-negative on a neighborhood of $q$ intersected with $\mathcal S$.
More precisely, for $q=(q_1,p_2)\in Z_1\times M_2$ there exist positive constants
$c_1(q), c_2(q), r_q$ such that
\begin{equation}\label{fine63}c_1(q)(x_2-p_2)^2-c_2(q)|x_1-q_1|\;\leq\;D_H(x)\;\leq\;
c_2(q)\big((x_2-p_2)^2+|x_1-q_1|\big),
\end{equation}
for $|x-q|<r_q$.
An analogous statement  holds at the  midpoints in
$M_1\times Z_2$.
\end{Lemma}
\begin{proof}
Thanks to (H1), near the point $q=(q_1,p_2)$ we have \[h_1(x_1)=h_1'(q_1)(x_1-q_1)+O((x_1-q_1)^2),\ \ \ \ \ \ h_1'(q_1)\neq 0,\]  so that
\[(h_1h_1'')(x_1)=O(x_1-q_1).\]
Moreover, in view of (H1) and (H2), we have
\[h_2(x_2)=h_2(p_2)+\frac 12 h_2^{\prime\prime}(p_2)(x_2-p_2)^2+O((x_2-p_2)^3),\ \ \ \ \ \ h_2(p_2)\neq 0,\]
and
\[h_2'(x_2)=h_2''(p_2)(x_2-p_2)+O((x_2-p_2)^2),\ \ \ \ \ \ h_2''(p_2)\neq 0,\] so that
\[(h_2')^2(x_2)=(h_2''(p_2))^2(x_2-p_2)^2+O((x_2-p_2)^3),\]
and
\[h_2(x_2) h^{\prime\prime}(x_2)=h_2(p_2)h^{\prime\prime}(p_2)+O(x_2-p_2).
\] 
All this yields
\[
D_H(x)=(h_1'(q_1))^2(h_2''(p_2))^2(x_2-p_2)^2+O\big(|x_1-q_1|+(x_2-p_2)^2|x_1-q_1|+(x_2-p_2)^3\big),
\]
giving \eqref{fine63} for $|x-q|<r_q$, for some $r_q>0$ sufficiently small.
\end{proof}


 	If we take, for instance, $H(x)=\sin x_{1}\cos x_{2}$ (see Example \ref{fine21}) and $\lambda(x)\equiv1$, then one can easily verify that 
\begin{equation*}
	\text{div}\ g_{\alpha}(x) = \frac{\alpha}{1+\alpha}\big(\cos^{2}x_{1}+\sin^{2}x_{2}-1\big),
\end{equation*}
whose level curves, when $\alpha=1$, are shown in Figure \ref{fig:div(g)}.  

\begin{figure}[htbp]
	\centering
	
	\begin{minipage}{0.9\textwidth}
		\centering
		\includegraphics[width=\linewidth]{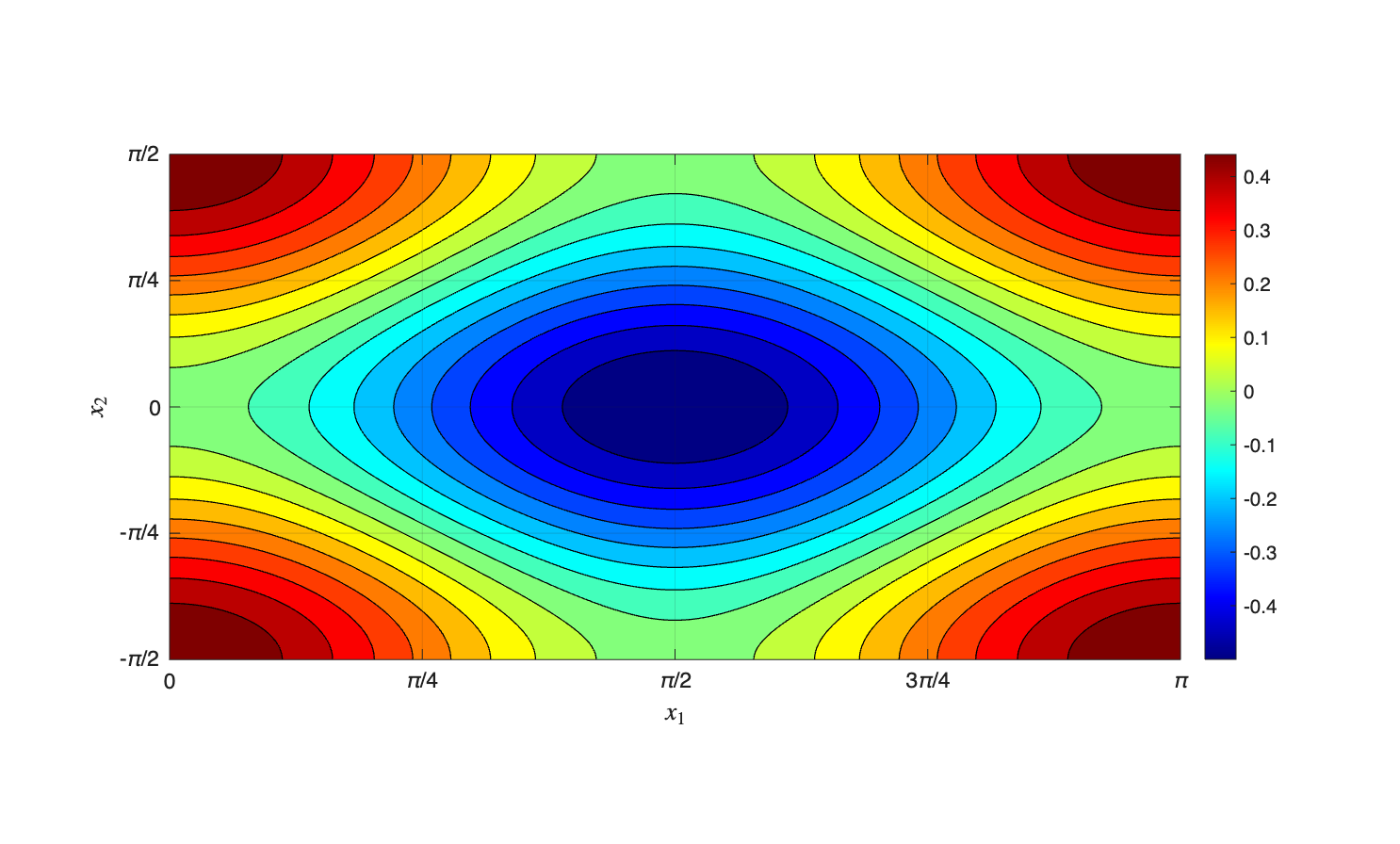}
		\caption{Level curves of $\mathrm{div}\, g_\alpha(x) = \tfrac{1}{2}(\cos^2 x_1 + \sin^2 x_2 - 1)$}
		\label{fig:div(g)}
	\end{minipage}
	
\end{figure}

\section{Monotonicity of $H$ along the flow and control of $g_\al$}\label{fine64}

In Theorem \ref{thm:main}, we assume that the initial density $f_0$ satisfies the following condition.
\begin{itemize}
\item[(S)] The support of $f_0$ does not contain any center point for $H$.	
\end{itemize}

In view of condition (S), we can find  $\mathfrak{r}>0$ small enough such that   the balls $B(p,\mathfrak{r})$ around the
centers $p\in M_1\times M_2$ are pairwise disjoint and
\begin{equation}\label{fine14}
\text{supp}\, f_0\;\subset\;D_{\mathfrak{r}}:=\T^2\setminus\bigcup_{p\in M_1\times M_2}B(p,\mathfrak{r}).
\end{equation}
In what follows, we shall denote
\begin{equation}\label{fine35}
\gamma_{\mathfrak{r}}:=\max_{x\in D_{\mathfrak{r}}}|H(x)|.
\end{equation}

\begin{Lemma}\label{fine36}
For every   $x \in\,\T^2$ we have $\Lambda(x)\geq 0$. Moreover, if $x \in\,D_{\mathfrak{r}}$  we have $\Lambda(x)>0$,  and
\[
\Lambda_{*}(\mathfrak{r}):=\min_{x\in D_{\mathfrak{r}}}\Lambda(x)>0.
\]
\end{Lemma}
\begin{proof}
Thanks to  (H3), we have $h_1h_1''\leq 0$ and $h_2h_2''\leq 0$ everywhere on $\T$, so that
\begin{equation}\label{fine66}\Lambda(x)\;\geq\;2(h_1')^2(x_1)(h_2')^2(x_2)\;\geq\;0\ \ \ \ \ \ \ \ x\in\T^2.
\end{equation}
In particular,\begin{equation}\label{fine67}
\{\Lambda=0\}\;\subset\;\{(h_1')^2(h_2')^2=0\}\;=\;(M_1\times\T)\cup (\T\times M_2).
\end{equation}

Within the level set $\{\Lambda=0\}$, 
we write $\Lambda(x)=-(h_1')^2(x_1) (h_2 h_2'')(x_2)-(h_1h_1'')(x_1)(h_2')^2(x_2).$ On $M_1\times\T$, we have 
$h_1'(x_1)=0$, so that $\Lambda(x)=-(h_1h_1'')(x_1)(h_2')^2(x_2)$. According to Proposition ~\ref{fine17} we have  $h_1h_1''<0$ on $M_1$, so that $\Lambda$ vanishes iff $h_2'=0$, i.e.\ at
$M_1\times M_2$. Symmetrically, we have
$\{\Lambda=0\}\cap(\T\times M_2)=M_1\times M_2$, and we conclude
 \[\{\Lambda=0\}=M_1\times M_2.\] 
 Finally, as all centers lie in
$\bigcup_p B(p,\mathfrak{r})$, we have $\Lambda>0$ on $D_{\mathfrak{r}}$, and the minimum 
attained on the compact set $D_{\mathfrak{r}}$ is strictly positive.
\end{proof}

Now, consider the stochastic flow
\begin{equation}\label{fine65}
d\phi_\al(t,x)=(\nu\,\xi)(\phi_\al(t,x))\circ dw(t)-g_\al(\phi_\al(t,x))\,dt,\quad
\phi_\al(0,x)=x \in\,\T^2,
\end{equation}
where $g_\al=\beta_\alpha\nu^2\,D\xi\xi$ is given by \eqref{fine11}.
We have
\[
dH(\phi_\al(t,x))=\nabla H(\phi_\al(t,x))\cdot(\nu\xi)(\phi_\al(t,x))\circ dw(t)
-\nabla H(\phi_\al(t,x))\cdot g_\al(\phi_\al(t,x))\,dt,
\]
and since $\nabla H(x)\cdot\xi(x)=\nabla H(x)\cdot\nabla^\perp H(x)=0$, the noise term
vanishes. Moreover, by \eqref{fine34},
\[
\nabla H(x)\cdot g_\al(x)=\beta_\alpha(x)\nu^2(x)\,\nabla H(x)\cdot (D\xi\xi)(x)=\beta_\alpha(x)\nu^2(x)\,H(x)\Lambda(x),
\]
so that 
\begin{equation}\label{fine12}dH(\phi_\al(t,x))=-\beta_\alpha(\phi_\al(t,x))\,\nu^2(\phi_\al(t,x))\,\Lambda(\phi_\al(t,x))\,H(\phi_\al(t,x))\,dt,
\end{equation}
and
\begin{equation}\label{fine13}H(\phi_\al(t,x))=H(x)\,\exp\!\bigg(\!-\!\int_0^t\beta_\alpha(\phi_\al(s,x))\nu^2(\phi_\al(s,x))\Lambda(\phi_\al(s,x))\,ds\bigg).
\end{equation}
Due to condition (A), we have
\[\beta_\alpha(x)\geq\beta_{\alpha, *}:=\frac{\al}{2(l^*+\al)},\ \ \ \ \ l^*:=\max_{x \in\,[-1,1]}l(x),\]
and
\[\nu^2(x)\geq m_{*}^2:=\min_{x \in\,[-1,1]}n^2(x)>0.\]
In view of Lemma \ref{fine36}, this means that  $|H(\phi_\al(t,x))|$ is non-increasing in $t$. Moreover,  \eqref{fine12} is a deterministic equation with random coefficients.
  
  \medskip

As a consequence of Lemma~\ref{fine36} and \eqref{fine13}, for every $x\in D_{\mathfrak{r}}$
and $t\geq 0$, \[
|H(\phi_\al(t,x))|\;\leq\;\gamma_{\mathfrak{r}}\,\exp\!\big(-\beta_{\alpha, *}\,m_{*}^2\,\Lambda_{*}(\mathfrak{r})\,t\big).
\]
In particular,  if for every 
$\eta\in(0,\gamma_{\mathfrak{r}})$ we define
\begin{equation}\label{fine68}
T_{\eta,\al}\;:=\;\frac{1}{\beta_{\alpha, *}\,m_{*}^2\,\Lambda_{*}(\mathfrak{r})}\,
\log\!\frac{\gamma_{\mathfrak{r}}}{\eta}\;=\;
\frac{2(l^{*}+\al)}{\al\,m_{*}^2\,\Lambda_{*}(\mathfrak{r})}\,\log\!\frac{\gamma_{\mathfrak{r}}}{\eta},
\end{equation}
we have
\begin{equation}\label{fine69}
\phi_\al(t,x)\in\,E_\eta:=\{|H|\leq\eta\},\ \ \ \ t\geq T_{\eta,\al},\ \ \ \ \ x\in D_{\mathfrak{r}}.
\end{equation}

\smallskip

For each midpoint $q\in\mathcal M$ , we fix a radius $r_q>0$ as in
Lemma~\ref{fine48}, and possibly shrink it so that the balls
$B(q,r_q)$, with $q\in\mathcal M$, are all pairwise disjoint, each contained
in a single cell-closure, and disjoint from the union of all $B(p,\mathfrak{r})$, with $p$ center.
For $\vt\in(0,\min_qr_q)$, we  define
\begin{equation}\label{fine71}
F_{\eta,\vt}\;:=\;\bigcup_{q\in\mathcal M}\big(E_\eta\cap B(q,\vt)\big),\qquad
G_{\eta,\vt}\;:=\;E_\eta\setminus F_{\eta,\vt}.
\end{equation}

\begin{Lemma}\label{fine72}There exist positive constants $\vt_1$ and $\eta_1$ and positive constants $c_1,c_2$, and $c_*0$, independent of $\eta_1$ and $\vt_1$, such that for every $\al>0$, 
and every $\vt\in(0,\vt_1)$ and $\eta\in(0,\eta_1\wedge c_*\vt^2)$,
\begin{equation}\label{sfm14}
x\in F_{\eta,\vt}\;\Longrightarrow\;|\div g_\al(x)|\leq
c_1\,\frac{\al\,n^2(0)}{l(0)+\al}\,\vt,
\end{equation}
and
\begin{equation}\label{sfm15}
x\in G_{\eta,\vt}\;\Longrightarrow\;\div g_\al(x)\geq
c_2\,\frac{\al\,n^2(0)}{l(0)+\al}\,\vt^2.
\end{equation}
\end{Lemma}

\begin{proof}
As in \eqref{fine52}, write $\div g_\al=T_1-T_2$, with
\[T_1(x):=\frac{\al}{l(H(x))+\al}\,n^2(H(x))\,D_H(x),\ \ \ \ \ \ \ T_2(x):=H(x)\,\Lambda(x)\,\Theta_\alpha(H(x)).\]

For every $h \in\,(0,\eta)$ and $\al>0$ we have\[
\bigg|\frac{\al/(l(h)+\al)}{\al/(l(0)+\al)}-1\bigg|\;=\;\frac{|l(0)-l(h)|}{l(h)+\al}\;\leq\;\frac{|l'|_\infty\eta}{l_0+\al}\;\leq\;\frac{|l'|_\infty\eta}{l_0}.
\]
Hence, if 
$\eta\leq \eta_{1,1}:=l_0/(2|l'|_\infty)$, we have
\begin{equation}\label{sfm6}
\frac{\al/(l(h)+\al)}{\al/(l(0)+\al)} \in\,[1/2,\,3/2].	
\end{equation}
Moreover, we have\[
\left|n^2(h)/n^2(0)-1\right|=\frac{|n^2(h)-n^2(0)|}{n^2(0)}=\frac{|n(h)-n(0)|\cdot|n(h)+n(0)|}{n^2(0)}\leq \frac{2|n|_\infty|n'|_\infty}{n^2(0)}\,\eta,
\]
and if we take 
 $\eta\leq \eta_{1,2}:=n^2(0)/(4 |n|_\infty|n'|_\infty)$, we get
 \begin{equation}
 n^2(h)/n^2(0) \in\,	[1/2,\,3/2].
 \end{equation}
Therefore, if we set
\begin{equation}\label{fine53}
\eta_1:=\;\eta_{1, 1}\wedge \eta_{1,2},
\end{equation}
we obtain\begin{equation}\label{fine54}
\frac14\,\frac{\al\,n^2(0)}{l(0)+\al}\;\leq\;\frac{\al}{l(h)+\al}\,n^2(h)\;\leq\;\frac94\,\frac{\al\,n^2(0)}{l(0)+\al},\ \ \ \ \ \ \ \ \eta\leq \eta_1.
\end{equation}
Finally, by Lemma \ref{fine48}, we have \[x\in F_{\eta,\vt}\subset
\bigcup_q B(q,\vt)\Longrightarrow |D_H(x)|\leq c_2(q)\,\vt,\ \ \ \ \ \ \vt<1,\]
and then, in view of \eqref{fine54}, we conclude that 
\begin{equation} \label{sfm2}
x\in F_{\eta,\vt}\Longrightarrow |T_1(x)|\leq \frac94\,\frac{\al\,n^2(0)}{l(0)+\al} \ \sup_{q \in\,\mathcal{M}} c_2(q)\,\vt=:\frac94\,\frac{\al\,n^2(0)}{l(0)+\al} C_{\mathcal{M}}\,\vt.
\end{equation}

Since $|\nabla H|$  vanishes only at the centers $M_1\times M_2$ and $D_{\mathfrak{r}}$ is compact, there exists some $\zeta_{\bar{r}}>0$ such that $|\nabla H(x)|\geq \zeta_{\bar{r}}$, for all $x \in\,D_{\mathfrak{r}}$. This implies that $E_\eta\cap D_{\mathfrak{r}}$ is contained in a tubular neighborhood of
$\mathcal S$ of width $\eta/\zeta_{\bar{r}}$.
If for every $x\in G_{\eta,\vt}\subset E_\eta\cap D_{\mathfrak{r}}$, we denote by  $x^*$  the
closest point of $\mathcal S$ to $x$, provided
\begin{equation}\label{fine55}
\eta\leq \zeta_{\bar{r}}\vt/2,
\end{equation}
the triangle inequality gives \[\dist(x^*,\mathcal M)\geq\dist(x,\mathcal M)-|x-x^*|\geq\vt-\vt/2=\vt/2.\]

Now, assume  $x^*\in\mathcal S\cap B(q,r_q)$, for some $q\in\mathcal M$.
We can take  $q=(q_1,p_2)\in Z_1\times M_2$, as the case $M_1\times Z_2$ can be treated analogously.
Near $q$, the separatrix is locally the vertical line $\{x_1=q_1\}$, so
$x^*=(q_1,p_2+\tau)$, for some $\tau\in\R$ with $|\tau|\leq r_q$. The
condition $\dist(x^*,\mathcal M)\geq\vt/2$ together with the disjointness
of all balls $B(q',r_{q'})$, with $q'\in\mathcal M$, forces $|\tau|\geq\vt/2$.
By Lemma~\ref{fine48}, with $x_1-q_1=0$ and $|x_2-p_2|=|\tau|$, we get
\begin{equation}\label{sfm5}
D_H(x^*)\;\geq\;c_q\,\tau^2\;\geq\;\frac{c_q}{4}\,\vt^2\geq \frac 14\,\min_{q\in\mathcal M}c_q\, \vt^2=:c_{\mathcal{M}}\,\vt^2.
\end{equation}
Next,  assume  $x^*\in\mathcal S\setminus\bigcup_{q\in\mathcal M}B(q,r_q)$.
Thanks again to Lemma  ~\ref{fine48}, $D_H>0$ on this set, which is
compact. Hence
\begin{equation}\label{sfm4}
D_H(x^*)\;\geq\;D_{m}\;:=\;\min_{y\in\mathcal S\setminus\bigcup_qB(q,r_q)}D_H(y)\;>\;0.
\end{equation}
Therefore, if we take 
\begin{equation}\label{fine73}
\vt_1\;\leq\;(D_{m}/c_{\mathcal M})^{1/2},
\end{equation}
combining together \eqref{sfm5} and \eqref{sfm4}, we obtain
\[
x\in G_{\eta,\vt}\Longrightarrow D_H(x^*)\;\geq\;c_{\mathcal M}\,\vt^2,\ \ \ \ \ \ \vt\leq \vt_1.\]

We now transfer the bound from $x^*$ to $x$ by continuity. Since $D_H\in C^1(\T^2)$, there exists $L>0$ (depending only on $h_1,h_2$) such that
$|D_H(x)-D_H(x^*)|\leq L\,|x-x^*|\leq L\eta/\zeta_{\bar{r}}$. Imposing
\begin{equation}\label{fine74}
L\eta/\zeta_{\bar{r}}\;\leq\;\,c_{\mathcal{M}}\,\vt^2/2,
\end{equation}
we conclude 
\[ D_H(x)\geq D_H(x^*)-L\eta/\zeta_{\bar{r}}\geq c_{\mathcal{M}}\,\vt^2-c_{\mathcal{M}}\,\vt^2/2=c_{\mathcal{M}}\,\vt^2/2,\] 
for all $x\in G_{\eta,\vt}$. This implies that 
\[x\in G_{\eta,\vt}\Longrightarrow 
T_1(x)\;\geq\;\frac{1}{4}\frac{\al\,n(0)^2}{l(0)+\al}c_{\mathcal{M}}\,\vt^2.
\]

Next, we bound $T_2(x)=H(x)\,\Lambda(x)\,\Theta_\alpha(H(x))$, with $\Theta_\alpha\in C^1([-1,1])$
given by \eqref{fine57}. Recalling that
\[
\Theta_\alpha(h)=\frac{\al\,l'(h)\,n^2(h)}{2(l(h)+\al)^2}-\frac{\al\,n(h)\,n'(h)}{l(h)+\al},
\]
we have\[
|\Theta_\alpha(h)|\;\leq\;\frac{\al}{l(h)+\al}\,\bigg(\frac{|l'|_\infty |n|_\infty^2}{2l_0}+|n|_\infty |n'|_\infty\bigg)\;=:\;\frac{\al}{l(h)+\al}\,c_{l, n}.
\]
Therefore, according to \eqref{sfm6},  for every  $x \in\,E_\eta$ and $\eta\leq \eta_1$
\begin{equation}\label{sfm7}
|T_2(x)|\;\leq\;\eta\,|\Lambda|_\infty\,\frac{\al}{l(H(x))+\al}\,c_{l, n}\leq \frac32\,\eta\,|\Lambda|_\infty\,\frac{\al}{l(0)+\al}c_{\lambda, \nu}=:\frac{\al\,n^2(0)}{l(0)+\al}\,c_{l, n}^\prime\,\eta.
\end{equation}

We define
\begin{equation}\label{fine58}
c_*\;:=\;\frac{\zeta_{\bar{r}}}{2}\wedge \frac{c_{\mathcal{M}} \zeta_{\bar{r}}}{2L}\wedge \frac{c_{\mathcal{M}}}{8c_{l, n}'},
\end{equation}
where the first two terms come from \eqref{fine55} and \eqref{fine74} and the third one from \eqref{sfm7}. With this
choice and $\eta\leq c_*\vt^2$,
\[
|T_2(x)|\;\leq\;c_{l, n}'\,\frac{\al\,n^2(0)}{l(0)+\al}\,\eta\;\leq\;c_{l, n}'\,c_*\,\frac{\al\,n^2(0)}{l(0)+\al}\,\vt^2\;\leq\;\frac{c_{\mathcal{M}}}{8}\,\frac{\al\,n^2(0)}{l(0)+\al}\,\vt^2.
\]
In particular,
\[
x\in F_{\eta,\vt}\;\Rightarrow\;|\div g_\al(x)|\leq |T_1(x)|+|T_2(x)|
\leq \frac 18\Big(18\,C_{\mathcal M}+{c_{\mathcal{M}}}\vt\Big)\frac{\al\,n^2(0)}{l(0)+\al}\vt
\leq c_1\,\frac{\al\,n^2(0)}{l(0)+\al}\,\vt,
\]
with $c_1:=(18\,C_{\mathcal M}+{c_{\mathcal{M}}}\vt_1)/8$, and
\[
x\in G_{\eta,\vt}\;\Rightarrow\;\div g_\al(x)\geq T_1(x)-|T_2(x)|\geq
\Big(\frac{c_{\mathcal{M}}}{4}-\frac{c_{\mathcal{M}}}{8}\Big)\frac{\al\,n^2(0)}{l(0)+\al}\vt^2
=\frac{c_{\mathcal{M}}}{8}\,\frac{\al\,n^2(0)}{l(0)+\al}\,\vt^2.
\]
Setting $c_2:=c_{\mathcal{M}}/8$, we  complete the proof.
\end{proof}

\section{Characteristic-flow representation}\label{fine56}

We introduce the stochastic flow $\phi_\alpha(t,x)$ defined by 
\begin{equation}
	d\phi_\alpha(t,x)= \nu(\phi_\alpha(t,x))\xi(\phi_\alpha(t,x))\circ dw(t)-g_{\alpha}(\phi_\alpha(t,x))dt,\ \ \ \ \phi_\alpha(0,x)=x.
\end{equation}
Next, we set 
\begin{equation}
	h_\alpha(t,x):=f_\alpha(t,\phi_\alpha(t,x)).
\end{equation}
By It\^{o}-Kunita-Wentzell formula (see \cite{fgp} for a proof in this framework), it holds
\begin{equation}
\begin{array}{l}
		\ds{  dh_\alpha(t,x) = (df_\alpha)(t,\phi_\alpha(t,x))+\nabla f_\alpha(t,x)\circ d\phi_\alpha(t,x) }\\[10pt]
		\ds{\quad \quad \quad \quad = \text{div}\,g_\alpha(\phi_\alpha(t,x))\,h_\alpha(t,x)\,dt-R_{0}h_\alpha^{2}(t,x)dt, \ \ \ \ \ h_\alpha(0,x)=f_{0}(x). }	\end{array}
\end{equation}
 Next, if we denote by $\psi_\alpha(t,x)$ the matrix $D\phi_\alpha(t,x)$, we have
\begin{equation}
	d\psi_\alpha(t,x) = D(\nu\xi)(\phi_\alpha(t,x) \psi_\alpha(t,x)\circ dw(t)-Dg_\alpha(\phi_\alpha(t,x))\psi_\alpha(t,x)dt,
\end{equation}
and, since $\div(\nu \xi)=0$, we have
\begin{equation}
	\begin{array}{ll}
	&\ds{ d\big(\text{det}\ \psi_\alpha\big)(t,x)=\text{det}\ \psi_\alpha(t,x)\big[ \text{div}(\nu \xi)(\phi_\alpha(t,x))\circ dw(t)-\text{div}\,g_\alpha(\phi_\alpha(t,x))\,dt \big] }\\
	\vs 
	&\ds{\quad\quad\quad\quad\quad\quad \quad\quad\quad  =-\text{div}g_\alpha(\phi_\alpha(t,x))(\text{det}\ \psi_\alpha)(t,x)\,dt },
	\end{array}
\end{equation}
with $\text{det}\ \psi_\alpha(0,x)=1$. This implies that
\begin{equation}
	\text{det}\, \psi_\alpha(t,x) = \exp\Big(-\int_{0}^{t}\text{div}\,g_\alpha(\phi_\alpha(s,x))\,ds\Big),
\end{equation}
so that
\begin{equation}
\begin{array}{l}
\ds{\frac{d}{dt}\left(h_\alpha(t,x)\exp\left(-\int_{0}^{t}\text{div}\,g_\alpha(\phi_\alpha(s,x))\,ds\right)\right)}\\[14pt]
\ds{\quad \quad  = -R_{0}\,\exp\left(\int_{0}^{t}\text{div}\,g_\alpha(\phi_\alpha(s,x))ds\right)\left(h_\alpha(t,x)\exp\left(-\int_{0}^{t}\text{div}\,g_\alpha(\phi_\alpha(s,x))\,ds\right)\right)^2.}	
\end{array}
	\end{equation}
This gives 
\begin{equation}
\begin{array}{l}
\ds{h_\alpha(t,x)\lvert \text{det}\ \psi_\alpha(t,x)\rvert= h_\alpha(t,x)\exp\left(-\int_{0}^{t}\text{div}\,g_\alpha(\phi_\alpha(s,x))ds\right)}\\[14pt]
\ds{ \quad \quad  \quad  \quad  \quad  \quad  \quad  = f_{0}(x)\left(1+R_0 f_{0}(x)\int_{0}^{t}\exp\left(\int_{0}^{s}\text{div}\,g_\alpha(\phi_\alpha(r,x))dr\right)  ds\right)^{-1},}	
\end{array}
	\end{equation}
so that
\begin{equation}\label{fin2}
\begin{array}{l}
\ds{\int_{\mathbb{T}^{2}}f_\alpha(t,x)dx = \int_{\mathbb{T}^{2}}h_\alpha(t,x)\lvert \text{det}\ \psi_\alpha(t,x)\rvert dx}\\[14pt]
\ds{ = \int_{\mathbb{T}^{2}}f_{0}(x)\left(1+R_0 f_{0}(x)\int_{0}^{t}\exp\left(\int_{0}^{s}\text{div}\,g_\alpha(\phi_\alpha(r,x))dr\right)  ds\right)^{-1}\,  dx.}	
\end{array}
	\end{equation}
	If we assume that $\text{ supp}\,f_0\subset D_{\mathfrak{r}}$ (see \eqref{fine14}), identity  \eqref{fin2} implies 	
\[\int_{\mathbb{T}^{2}}f_\alpha(t,x)dx\leq \frac 1{R_0}\int_{D_{\mathfrak{r}}}\left(\int_0^t\exp\Big(\int_0^s \text{div}\,g_\alpha(\phi_\alpha(r,x))\,dr\Big)\,ds\right)^{-1}\,dx.\]
Since $\text{div}\,g_\alpha\leq 1$, we have
\[\begin{array}{l}
\ds{\int_0^t\exp\Big(\int_0^s \text{div}\,g_\alpha(\phi_\alpha(r,x))\,dr\Big)\,ds}\\[14pt]
\ds{=\exp\Big(\int_0^t \text{div}\,g_\alpha(\phi_\alpha(r,x))\,dr\Big)\int_0^t\exp\Big(-\int_s^t \text{div}\,g_\alpha(\phi_\alpha(r,x))\,dr\Big)\,ds}\\[14pt]
\ds{\geq \exp\Big(\int_0^t \text{div}\,g_\alpha(\phi_\alpha(r,x))\,dr\Big)\int_0^te^{-(t-s)}\,ds= \exp\Big(\int_0^t \text{div}\,g_\alpha(\phi_\alpha(r,x))\,dr\Big)(1-e^{-t}). }	
\end{array}\]
Thus, if we take   $t\geq \log 2$ we obtain\begin{equation}
\label{fin20}
\int_{\mathbb{T}^2}f_\alpha(t,x)\,dx\leq \frac 2{R_0}\int_{D_{\mathfrak{r}}}\exp\Big(-\int_0^t \text{div}\,g_\alpha(\phi_\alpha(r,x))\,dr\Big)\,dx.	
\end{equation}


 To illustrate the joint behavior of $H$ and $\mathrm{div}\, g_{\alpha}$, Figure \ref{fig:joint} displays the level curves of $H(x) = \sin x_1 \cos x_2$ and $\mathrm{div}\, g_\alpha(x)$ simultaneously. The function $\mathrm{div}\, g_\alpha$ is strictly negative at the center  $(\pi/2, 0)$, where $|H|$ attains its maximum, and transitions to positive values as $x$ approaches the separatrix $\{H = 0\}$.  The red curve marks the zero set $\{\mathrm{div}\, g_\alpha = 0\}$, which separates the region where the inertial drift compresses the flow from the region where it expands it. The gray shaded corners indicate the region $\{\mathrm{div}\, g_\alpha \geq 0.3\}$, lying near the separatrix, where the drift most strongly concentrates mass toward $\{H = 0\}$.

\begin{figure}[htbp]
	\centering
	
	\begin{minipage}{0.9\textwidth}
		\centering
		\includegraphics[width=\linewidth]{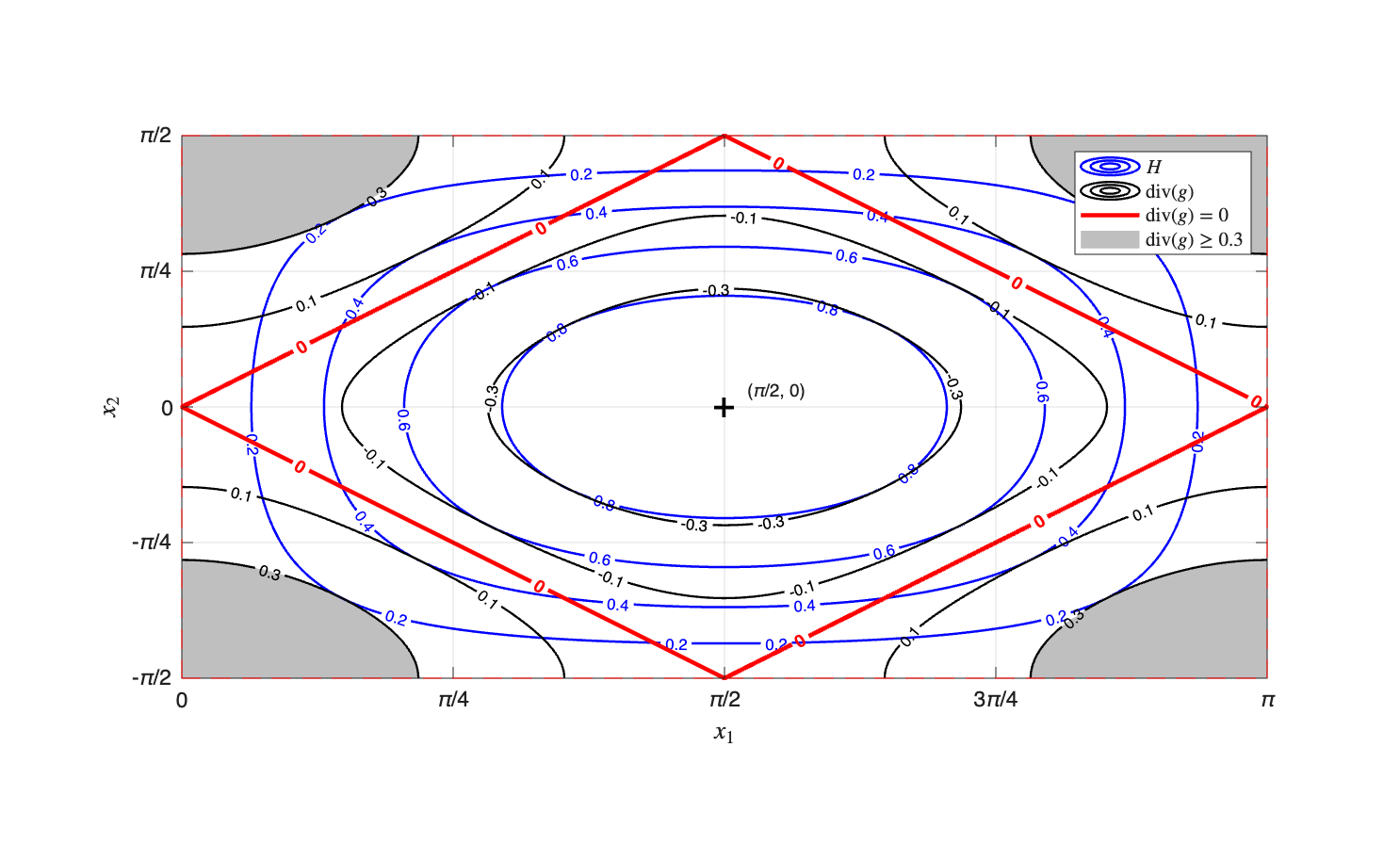}
		\captionsetup{justification=centering}
		\caption{Joint level curves of $H(x) = \sin x_1 \cos x_2$ (blue) and 
			$\mathrm{div}\, g_\alpha(x)$ (black) on the region $[0,\pi]\times[-\pi/2,\pi/2]$}
		\label{fig:joint}
	\end{minipage}
	
\end{figure}


Combining \eqref{fin20} with \eqref{fine69}, due to \eqref{sfm12} for every
$t\geq T_{\eta,\al}$ we obtain
\[\begin{array}{l}
\ds{\int_{\T^2}f_\alpha(t,x)\,dx\leq \frac{c}{R_0}\exp\left(\frac{c\, \alpha}{ l_0+\alpha}\, T_{\eta, \alpha}\right)\int_{D_{\mathfrak{r}}}\exp\Big(-\int_{T_{\eta, \alpha}}^t \text{div}\,g_\alpha(\phi_\alpha(r,x))\,dr\Big)\,dx}	\\[16pt]
\ds{\quad \quad \quad \quad \quad \quad \quad \quad \quad \quad \leq \frac{c}{R_0}\, \eta^{-\kappa}\int_{D_{\mathfrak{r}}}\exp\Big(-\int_{T_{\eta, \alpha}}^t \text{div}\,g_\alpha(\phi_\alpha(r,x))\,dr\Big)\,dx,}
\end{array}
\]
for some constant $\kappa>0$ independent of $\alpha>0$.
Thus, if we set $k(t):=\lfloor t-T_{\eta, \alpha}\rfloor$, and $s_j:=T_{\eta, \alpha}+j$, for $j=0,\ldots,k-1$, and define
\[J_{\alpha}(s,x):=\exp\Big(-\int_{s}^{s+1} \text{div}\,g_\alpha(\phi_\alpha(r,x))\,dr\Big),\ \ \ \ \ \ x \in\,D_{\mathfrak{r}},\]
we have
\begin{equation}\label{new17}
\begin{array}{l}\ds{\int_{\T^2}f_\alpha(t,x)\,dx \leq \frac{c}{R_0}\, \eta^{-\kappa}\,\int_{D_{\mathfrak{r}}}\prod_{j=0}^{k(t)-1} J_{\alpha}(s_j,x)\,dx.}
\end{array}	
\end{equation}

Now, assume that the following result holds.
\begin{Proposition}\label{prop7.1}
There exist $\eta_0,\vt_0>0$,  independent of   $R_0$ and $f_0$,  and a constant $\bar\kappa>0$ such that, for
every $\al>0$, every $x\in D_{\mathfrak{r}}$, and every $s\geq T_{\eta_0,\al}$,
\begin{equation}\label{fine15}
\E\big(J_\al(s,x)\,|\,\mathcal F_s\big)\leq\exp\!\bigg(\!-\frac{\al\,n^2(0)}{l(0)+\al}\,\bar\kappa\bigg),\ \ \ \ \ \ \ \ \ \ \P\text{-a.s.}
\end{equation}
\end{Proposition}

As a consequence of the above result, if we set
\begin{equation}
	\label{sfm20}
	\chi_\alpha:=\frac{\al\,n^2(0)}{l(0)+\al},
\end{equation}
 we have
\[\begin{array}{l}
\ds{\mathbb{E}\,\prod_{j=0}^{k(t)-1}J_{\alpha}(s_j,x)=\mathbb{E}\Big(\prod_{j=0}^{k(t)-2}J_{\alpha}(s_j,x)\,\mathbb{E}\left(J_{\alpha}(s_{k(t)},x)|\mathcal{F}_{s_{k(t)}}\right)\Big)\leq e^{-\chi_\alpha\bar{\kappa}}\,\mathbb{E}\,\prod_{j=0}^{k(t)-2}J_{ \alpha}(s_j,x),}
\end{array}\]
and a recursive argument allows us to conclude that
\[\mathbb{E}\,\prod_{j=0}^{k(t)-1}J_{\alpha}(s_j,x)\leq e^{- \chi_\alpha\bar{\kappa}\,k(t)}\leq e^{(T_{{\eta_0}, \alpha}+1)\chi_\alpha\bar{\kappa}}\,e^{-\chi_\alpha\bar{\kappa}\, t},\ \ \ \ \ x \in\,D_{\mathfrak{r}}.\]
Thanks to \eqref{new17}, this implies 
\[\mathbb{E}\int_{\T^2}f_\alpha(t,x)\,dx \leq \frac{c}{R_0}\, {\eta}^{-\kappa}\,\exp\left(\chi_\alpha\,\bar\kappa\right)e^{(T_{\bar{\eta}, \alpha}+1)\chi_\alpha\bar\kappa}\,e^{-\chi_\alpha\bar\kappa\, t},\]
and \eqref{fin1} follows. Thus, in order to complete the proof of Theorem \ref{thm:main}, we will only need to prove Proposition \ref{prop7.1}.

\section{The flow on a cell: It\^o form and non-degeneracy near midpoints}\label{fine26}

As we have already seen, the flow $\phi_\al$ leaves each cell $\mathcal C$ of
$\T^2\setminus\mathcal S$ invariant. Namely, for every cell $\mathcal C$ of $\T^2\setminus\mathcal S$ and every
$x\in\mathcal C$, it holds $\phi_\al(t,x)\in\overline{\mathcal C}$ for all $t\geq 0$,
$\P\text{-a.s.}$
Indeed, if $H(x)>0$, then by
\eqref{fine13} $H(\phi_\al(t,x))>0$, for all $t\geq 0$, so that $\phi_\al(t,x)$
remains in the connected component of $\{H>0\}$ containing $x$ — which is
the cell $\mathcal C$ containing $x$. The same holds for $H(x)<0$. 

In what follows we fix a single cell $\mathcal C$ and  $x\in D_{\mathfrak{r}}\cap\mathcal C$.
The constants we get will depend on the cell, but since there are finitely many
cells, we eventually take the worst constant. Lifting $\overline{\mathcal C}$ to a closed
rectangle $[a_1,b_1]\times[a_2,b_2]\subset\R^2$ with $b_i-a_i<2\pi$ and
$h_i(a_i)=h_i(b_i)=0$, we work with $\phi_\al$ as an $[a_1,b_1]\times[a_2,b_2]$-valued continuous semimartingale.

\subsection{It\^o form of the flow }

\smallskip

Write $\phi_\al=(\phi_{\al,1},\phi_{\al,2})$. The Stratonovich-It\^o
conversion gives, for any $C^2$ vector field $V$,
\[
V(\phi_\al(t,x))\circ dw(t)\;=\;V(\phi_\al(t,x))\,dw(t)\,+\,\tfrac12\,DV(\phi_\al(t,x))\,V(\phi_\al(t,x))\,dt.
\]
For $V=\nu\xi$, we computed in the proof of Lemma~\ref{fine10}
that $D(\nu\xi)(\nu\xi)=\nu(\xi\cdot\nabla\nu)\xi+\nu^2\,D\xi\,\xi$, and
under (A) the first term vanishes. Hence $DV\,V=\nu^2\,D\xi\,\xi$.
Combined with \eqref{fine65} and $g_\al=\beta_\alpha\,\nu^2\,D\xi\,\xi$, this yields the It\^o form
\begin{equation}\label{fine28}
d\phi_\al(t,x)\;=\;\nu(\phi_\al(t,x))\,\xi(\phi_\al(t,x))\,dw(t)\,+\,\Big(\tfrac12-\beta_\alpha(\phi_\al(t,x))\Big)\nu(\phi_\al(t,x))^2\,D\xi(\phi_\al(t,x))\,\xi(\phi_\al(t,x))\,dt.
\end{equation}
Componentwise,
\[
\sigma_{\al,1}(x):=\nu(x)\,\xi_1(x)=-\nu(x)\,h_1(x_1)\,h_2'(x_2),\quad
\sigma_{\al,2}(x):=\nu(x)\,\xi_2(x)=\nu(x)\,h_1'(x_1)\,h_2(x_2),
\]
and
\[
b_{\al,i}(x)\;:=\;\Big(\tfrac12-\beta_\alpha(x)\Big)\,\nu(x)^2\,(D\xi(x)\,\xi(x))_i,\qquad i=1,2.
\]

Since $\al>0$ and $l\geq l_0>0$, we have 
\[\beta_\alpha(x)=\al/(2(l(H(x))+\al))\in(0,1/2),\]
so that  $\tfrac12-\beta_\alpha(x)\in(0,1/2)$. Together with the bound $|\nu|\leq\vert n\vert_\infty<\infty$, this gives
\[
|b_{\al,i}(x)|\;\leq\;\tfrac12\,\vert n\vert_\infty^2\,\vert D\xi\,\xi\vert_\infty,
\]
and the right-hand side is a constant depending only on the $C^2$-norms of $h_1,h_2$ and on $m$ (independent of $\al$). Set
\begin{equation}\label{fine29}c_b\;:=\;\tfrac12\,\vert n\vert_\infty^2\,\vert D\xi\,\xi\vert _\infty,\qquad
c_\sigma\;:=\;\vert n\vert _\infty\,\max\big(\vert h_1\vert _\infty\vert h_2'\vert _\infty,\vert h_1'\vert _\infty\vert h_2\vert _\infty\big).
\end{equation}
Both constants are independent of $\al$ and depend only on the $C^2$-norms of $h_1,h_2$ and on $m$ and 
for all $\alpha>0$ 
\begin{equation}\label{sfm100}
|\sigma_{\alpha, i}(x)|\leq c_\sigma,\ \ \ \ \ \ |b_{\alpha, i}(x)|\leq c_b,\ \ \ \ \ i=1,2.	
\end{equation}

\subsection{Non-degeneracy estimates near midpoints}

The localization argument relies on the following structural fact.

\begin{Lemma}\label{fine32}
There exist constants $r_*>0$ and $c_1^*>0$, depending only on $h_1$,
such that
\[\dist(x_1,Z_1)\leq r_*\Longrightarrow
|h_1'(x_1)|\;\geq\;c_1^*.
\]
Moreover, there exists $m_*>0$, depending only on $h_1$ and $r_*$, such
that
\[\dist(x_1,Z_1)\geq r_*\Longrightarrow 
|h_1(x_1)|\;\geq\;m_*.
\]
The symmetric statements hold for $h_2$, with constants $c_2^*,m_*'>0$.
\end{Lemma}

\begin{proof}
For the first claim, thanks to (H1) we have $|h_1'(q)|>0$ for every $q\in Z_1$. By
continuity of $h_1'$, there exists $r_*>0$ such that
$|h_1'(x_1)|\geq|h_1'(q)|/2$ on $[q-r_*,q+r_*]$ for every $q\in Z_1$.
Shrinking $r_*$ if necessary so that the intervals
$\{[q-r_*,q+r_*]\}_{q\in Z_1}$ are pairwise disjoint (possible since
$Z_1$ is finite), we define \[c_1^*:=\min_{q\in Z_1}|h_1'(q)|/2>0.\]

For the second claim, we have that $\{x_1\in\T:\dist(x_1,Z_1)\geq r_*\}$ is a compact
set on which $h_1$ does not vanish and by continuity
$|h_1|$ attains a positive minimum $m_*$ on it.
\end{proof}

We conclude with the following localization result.

\begin{Lemma}
   \label{fine37}
There exist $\eta_0\in(0,\eta_1]$ and, for each midpoint
$q\in\mathcal M$, an open neighborhood $V_q\subset\T$ of the tangential
coordinate of $q$ and a constant $c_*(q)>0$, all depending only on
$h_1,h_2$, such that the following holds.
For every $\eta\in(0,\eta_0]$, every $\al>0$, every $r\geq 0$, every
cell $\mathcal C$ with $q\in\overline{\mathcal C}$, and every
$x\in D_{\mathfrak{r}}\cap\mathcal C$
\begin{itemize}
\item[1.] If $q=(q_1,p_2)\in Z_1\times M_2$ and
$\phi_\al(r,x)\in E_\eta$, with $\phi_{\al,2}(r,x)\in V_q$, then
\[
|\sigma_{\al,2}(\phi_\al(r,x))|^2\;\geq\;c_*(q).
\]
\item[2.] If $q=(p_1,q_2)\in M_1\times Z_2$ and
$\phi_\al(r,x)\in E_\eta$ with $\phi_{\al,1}(r,x)\in V_q$, then
\[
|\sigma_{\al,1}(\phi_\al(r,x))|^2\;\geq\;c_*(q).
\]
\end{itemize}
\end{Lemma}

\begin{proof}
We prove only (i), as (ii) is symmetric. Fix $q=(q_1,p_2)\in Z_1\times M_2$.
 Since
$p_2\in M_2$ and $M_2\cap Z_2=\emptyset$, $h_2(p_2)\neq 0$. Set
$c_2(q):=|h_2(p_2)|/2>0$, and let $V_q\subset\T$ be the open connected
component of $\{x_2\in\T:|h_2(x_2)|>c_2(q)\}$ containing $p_2$. Then
$|h_2|\geq c_2(q)$ on $V_q$.

Now, let
$r_*,c_1^*$ and $m_*$ be the constants of Lemma~\ref{fine32} for
$h_1$, and set
\begin{equation}\label{fine40}
\eta_0(q)\;:=\;c_2(q)\,m_*.
\end{equation}
Suppose $\eta\leq\eta_0(q)$, $\phi_\al(r,x)\in E_\eta$, and
$\phi_{\al,2}(r,x)\in V_q$. Then
$|h_2(\phi_{\al,2}(r,x))|\geq c_2(q)$, so that
\begin{equation}\label{fine41}|h_1(\phi_{\al,1}(r,x))|\;=\;\frac{|H(\phi_\al(r,x))|}{|h_2(\phi_{\al,2}(r,x))|}\;\leq\;\frac{\eta}{c_2(q)}\;\leq\;\frac{\eta_0(q)}{c_2(q)}\;=\;m_*.
\end{equation}
In view of  the second part of Lemma~\ref{fine32}, this implies 
$\dist(\phi_{\al,1}(r,x),Z_1)<r_*$, and then, by the first part of
Lemma~\ref{fine32}, we conclude $|h_1'(\phi_{\al,1}(r,x))|\geq c_1^*$.
Therefore, 
\[
|\sigma_{\al,2}(\phi_\al(r,x))|\;=\;\nu(\phi_\al(r,x))\,|h_1'(\phi_{\al,1}(r,x))|\,|h_2(\phi_{\al,2}(r,x))|\;\geq\;m_{*}\,c_1^*\,c_2(q).
\]
Finally, if we set
\[c_*(q):=\sqrt{m_{*}\,c_1^*\,c_2(q)},\ \ \ \ \ \ \eta_0:=\eta_1\wedge\min_{q\in\mathcal M}\eta_0(q),\]
we conclude the proof of the lemma.
\end{proof}

   \section{Proof of Proposition \ref{prop7.1}}\label{sec8}
We start from the following averaging estimate for $\div g_\alpha$.
\begin{Lemma}\label{fine59}
Fix $\vt\in(0,\vt_1)$ and $\eta\in(0,\min(\eta_1,c_*\vt^2))$. Then, for every
$\al>0$, $x\in D_{\mathfrak{r}}$, and  $s\geq T_{\eta,\al}$, we have
\begin{equation}\label{fine60}
\int_s^{s+1}\div g_\al(\phi_\al(r,x))\,dr\;\geq\;
\chi_\alpha\bigg(\,c_2\,\vt^2-(c_2\vt^2+c_1\vt)\int_s^{s+1}\ind_{F_{\eta,\vt}}(\phi_\al(r,x))\,dr\bigg),\quad\P\text{-a.s.}
\end{equation}
where $\chi_\alpha$ is defined as in \eqref{sfm20}
\end{Lemma}

\begin{proof}
Since $s\geq T_{\eta,\al}$, we have that $\phi_\al(r,x)\in E_\eta$, for all $r\geq s$.
Splitting the integral on $E_\eta$ in two parts, one on $F_{\eta,\vt}$ and one on $G_{\eta,\vt}$, we write
\[\begin{array}{l}\ds{
\int_s^{s+1}\div g_\al(\phi_\al(r,x))\,dr=\int_s^{s+1}\div g_\al(\phi_\al(r,x))\,\ind_{F_{\eta,\vt}}(\phi_\al(r,x))\,dr}\\[10pt]
\ds{\quad \quad \quad \quad \quad +\int_s^{s+1}\div g_\al(\phi_\al(r,x))\,\ind_{G_{\eta,\vt}}(\phi_\al(r,x))\,dr=:I_{\alpha, 1}(s)+I_{\alpha, 2}(s).}\end{array}
\]
According to \eqref{sfm14} and \eqref{sfm15}, respectively,  if we define $\chi_\alpha$ as in \eqref{sfm20},   we have
\[I_{\alpha, 1}(s)\geq -c_1\chi_\alpha\,\vt\int_s^{s+1}\ind_{F_{\eta,\vt}}(\phi_\al)\,dr,\]
and 
\[I_{\alpha, 2}(s)\geq c_2\chi_\a\,\vt^2\int_s^{s+1}\ind_{G_{\eta,\vt}}(\phi_\al)\,dr.\]
Then, by using $\ind_{G_{\eta,\vt}}=1-\ind_{F_{\eta,\vt}}$ on $E_\eta$, we obtain 
\eqref{fine60}.
\end{proof}

Now, raising to the exponential \eqref{fine60}, and using the inequality 
\[e^{\gamma z}\leq 1+(e^\gamma-1)z,\ \ \ \ z\in[0,1],\ \ \ \gamma>0,\]
we get
\begin{equation}\label{fine61}
J_\al(s,x)\leq\exp\!\Big(\!-c_2\chi_\a\,\vt^2\Big)
\bigg(1+\Big(\!\exp\big((c_2\vt^2+c_1\vt)\chi_\alpha\big)-1\Big)\!\int_s^{s+1}\ind_{F_{\eta,\vt}}(\phi_\al(r,x))\,dr\bigg).
\end{equation}
Now, for $\vt\leq c_1/c_2$, we have  $c_2\vt^2\leq c_1\vt$, so the inner exponent in \eqref{fine61}
is at most $2c_1\chi_\alpha\,\vt$.
Thus, by taking the conditional expectation given $\mathcal F_s$, we have
\begin{equation}\label{fine62}
\begin{array}{l}
\ds{\E\big(J_\al(s,x)|\mathcal F_s\big)}\\[10pt]
\ds{\quad \quad \leq\exp\left(-c_2\chi_\alpha\,\vt^2\right)
\left(1+\left(\exp\left(2c_1\chi_\al\,\vt\right)-1\right)\E\Big(\int_s^{s+1}\ind_{F_{\eta,\vt}}(\phi_\al(r,x))\,dr\,\big|\mathcal F_s\Big)\right).}
\end{array}
\end{equation}

The remaining task consists of bounding the conditional expectation
\[\E\left(\int_s^{s+1}\ind_{F_{\eta,\vt}}(\phi_\al(r,x))\,dr\,|\,\mathcal F_s\right).\]

\medskip

We start with the following lemma that provides a fundamental occupation-time estimate near a midpoint.

\begin{Lemma}\label{fine75}
There exists $\eta_2\in(0,\eta_0)$ and $\vt_2\in(0,\vt_1)$, and for each
midpoint $q\in\mathcal M$ a constant $\kappa_0(q)>0$, independent of  $\al,\eta,\vt,s$ and $x$, such that the
following holds.

For every $\al>0$, $\vt\in(0,\vt_2)$, $\eta\in(0,\min(\eta_2,c_*\vt^2))$,
$s\geq T_{\eta,\al}$, $x\in D_{\mathfrak{r}}$, and $q=(q_1,b_2)\in Z_1\times M_2$
\begin{equation}\label{fine77}
\E\Big(\int_s^{s+1}\ind_{[b_2-\vt,b_2+\vt]}(\phi_{\al,2}(r,x))\,dr\,\big|\,\mathcal F_s\Big)\;\leq\;\kappa_0(q)\,\vt,\ \ \ \ \ \ \ \ \ \P\text{-a.s.}
\end{equation}
The symmetric statement for $q=(p_1,q_2)\in M_1\times Z_2$, with
$\phi_{\al,1}$ in place of $\phi_{\al,2}$ and $[p_1-\vt,p_1+\vt]$ in
place of $[b_2-\vt,b_2+\vt]$, also holds.
\end{Lemma}

\begin{proof}
We prove only \eqref{fine77}, as the other case is symmetric.
By cell invariance, $\phi_\al(r,x)\in\overline{\mathcal C}$, for
all $r\geq 0$, where $\mathcal C$ is the cell containing $x$. If
$q\notin\overline{\mathcal C}$, then $\ind_{[b_2-\vt,b_2+\vt]}(\phi_{\al,2}(r,x))$ vanishes
for $\vt$ small enough, and  the bound is
trivial. Hence we may assume $q\in\overline{\mathcal C}$, and lift $\phi_{\al,2}$ to a
continuous $[a_2,b_2]$-valued semimartingale.

\smallskip
\noindent\emph{Step 1.} For each
$a\in\R$, Tanaka's formula applied to the continuous semimartingale
$\phi_{\al,2}(\cdot,x)$ at level $a$ gives, for $t\geq s$,
\begin{equation}\label{{fine78}}
(\phi_{\al,2}(t,x)-a)^+=(\phi_{\al,2}(s,x)-a)^++\!\int_s^t\!\ind_{\{\phi_{\al,2}(r,x)>a\}}d\phi_{\al,2}(r,x)+\tfrac12(L^a_t-L^a_s),
\end{equation}
where $L^a$ is the local time of $\phi_{\al,2}(\cdot,x)$ at level $a$.
Since $|\sigma_{\al,2}|\leq c_\sigma$ ( see
Section~\ref{fine26}), the stochastic integral
\[\int_s^t\ind_{\{\phi_{\al,2}(r,x)>a\}}\sigma_{\al,2}(r,x)\,dw(r)\] is a true
martingale, and taking conditional expectation given $\mathcal F_s$ in
\eqref{{fine78}}, for $t=s+1$,
\begin{equation}\begin{array}{l}\label{fine78}
\ds{\frac12\,\E(L^a_{s+1}-L^a_s|\mathcal F_s)
=\E\big((\phi_{\al,2}(s+1,x)-a)^+
-(\phi_{\al,2}(s,x)-a)^+\,\big|\,\mathcal F_s\big)}\\[10pt]
\ds{\quad \quad \quad \quad \quad \quad \quad \quad \quad \quad -\E\Big(\int_s^{s+1}\ind_{\{\phi_{\al,2}(r,x)>a\}}b_{\al,2}(\phi_\al(r,x))\,dr\,\big|\,\mathcal F_s\Big).}
\end{array}\end{equation}
Since $r\mapsto(r-a)^+$ is $1$-Lipschitz and $\phi_{\al,2}(\cdot,x)\in[a_2,b_2]$, due to cell invariance, with
$b_2-a_2<2\pi$, we have
\[
\big|(\phi_{\al,2}(s+1,x)-a)^+-(\phi_{\al,2}(s,x)-a)^+\big|\leq
|\phi_{\al,2}(s+1,x)-\phi_{\al,2}(s,x)|\leq b_2-a_2<2\pi.
\]
Hence, since  $|b_{\al,2}|\leq c_b$ (see  Section \ref{fine26}), from \eqref{fine78}, we conclude
\begin{equation}\label{fine79}
\sup_{a\in\R}\,\E(L^a_{s+1}-L^a_s|\mathcal F_s)\leq 2(b_2-a_2)+2cc_b\;=:\;\kappa_{\text{loc}},\ \ \ \ \ \ \ \P\text{-a.s.}
\end{equation}

\smallskip
\noindent\emph{Step 2.} For the continuous
semimartingale $\phi_{\al,2}(\cdot,x)$, the following occupation-time formula holds
\begin{equation}\label{fine80}
\int_s^{s+1}\ind_{[b_2-\vt,b_2+\vt]}(\phi_{\al,2}(r,x))\,\sigma_{\al,2}^2(\phi_\al(r,x))\,dr=\int_{b_2-\vt}^{b_2+\vt}(L^a_{s+1}-L^a_s)\,da.
\end{equation}
Taking the conditional expectation given $\mathcal F_s$ and using
\eqref{fine79},
\begin{equation}\label{fine76}
\E\Big(\int_s^{s+1}\ind_{[b_2-\vt,b_2+\vt]}(\phi_{\al,2}(r,x))\,\sigma_{\al,2}(r,x)^2\,dr\,\big|\,\mathcal F_s\Big)\leq 2\kappa_{\text{loc}}\,\vt.
\end{equation}

\smallskip
\noindent\emph{Step 3.}
For $\vt\leq\vt_2$ small enough that $[b_2-\vt,b_2+\vt]\subset V_q$, where $V_q$ is  the
neighborhood from Lemma~\ref{fine37}, we have
that whenever $r\in[s,s+1]$ is such that
$\phi_{\al,2}(r,x)\in[b_2-\vt,b_2+\vt]$ and $\phi_\al(r,x)\in E_\eta$, with $\eta\leq\eta_0$,  then $|\sigma_{\al,2}(\phi_\al(r,x))|^2\geq c_*(q)$ (see Lemma~\ref{fine37}).

Now, since $s\geq T_{\eta,\al}$, then thanks \eqref{fine69} we have that $\phi_\al(r,x)\in E_\eta$, for all $r\geq s$. Hence
\[
\ind_{[b_2-\vt,b_2+\vt]}(\phi_{\al,2}(r,x))\;\leq\;\frac{1}{c_*(q)}\,
\ind_{[b_2-\vt,b_2+\vt]}(\phi_{\al,2}(r,x))\,\sigma_{\al,2}^2(\phi_\al(r,x)),\ \ \ \ \ \ r\in[s,s+1].
\]
Thus, integrating, taking the conditional expectation given $\mathcal F_s$, and
applying \eqref{fine76}, we conclude
\[
\E \Big(\int_s^{s+1}\ind_{[b_2-\vt,b_2+\vt]}(\phi_{\al,2}(r,x))\,dr\,\big|\,\mathcal F_s\Big)\leq\frac{2\kappa_{\text{loc}}\vt}{c_*(q)}\;=:\;\kappa_0(q)\,\vt,
\]
which is \eqref{fine77}.
\end{proof}

\begin{Remark}
{\em Note that $\kappa_{\text{loc}}$ and $c_*(q)$ depend only on $h_1,h_2,l,r$ and 
in particular, the constant $\kappa_0(q)$ is uniform in $\al>0$. The
$\al$-dependence of the bound \eqref{fine62} therefore comes
entirely from the prefactor $\chi_\al$.}
\end{Remark}

Finally, we can conclude the proof of Proposition \ref{prop7.1}, with the following bound on the time spent in $F_{\eta,\vt}$.

\begin{Lemma}\label{fine38}
There exists $\kappa_0>0$ depending only on $h_1,h_2,l,r$ such that,
for every $\al>0$, $\vt\in(0,\vt_2)$,
$\eta\in(0,\min(\eta_2,c_*\vt^2))$, $s\geq T_{\eta,\al}$, and $x\in D_{\mathfrak{r}}$,
\begin{equation}\label{fine39}\E \Big(\int_s^{s+1}\ind_{F_{\eta,\vt}}(\phi_\al(r,x))\,dr\,\big|\,\mathcal F_s\Big)\leq\;\kappa_0\,\vt,\ \ \ \ \ \ \ \P\text{-a.s.}
\end{equation}
\end{Lemma}

\begin{proof}
 For each midpoint
$q=(q_1,b_2)\in Z_1\times M_2$ we have
\[ E_\eta\cap B(q,\vt)\subset
\{x:x_2\in[b_2-\vt,b_2+\vt]\}.\] Hence
\[
\ind_{E_\eta\cap B(q,\vt)}(\phi_\al(r,x))\leq
\ind_{[b_2-\vt,b_2+\vt]}(\phi_{\al,2}(r,x)).
\]
Symmetrically, the same is trus  for $q=(p_1,q_2)\in M_1\times Z_2$,  with the roles of $1$
and $2$ swapped. Summing and applying Lemma~\ref{fine75},
\[
\E \Big(\int_s^{s+1}\ind_{F_{\eta,\vt}}(\phi_\al(r,x))\,dr\,\big|\,\mathcal F_s\Big)
\leq\sum_{q\in\mathcal M}\kappa_0(q)\,\vt\;=\;|\mathcal M|\,\max_{q \in\,\mathcal{M}}\kappa_0(q)\cdot\vt\;=:\kappa_0\vt.
\]
\end{proof}

We are now in position to prove \eqref{fine15}.
Fix $\vt\in(0,\vt_2)$ to be chosen explicitly below, and set
$\eta:=c_*\vt^2$. The constraint $\eta\leq\min(\eta_2,c_*\vt^2)$ from the
preceding lemmas is then equivalent to $c_*\vt^2\leq\eta_2$, which holds
provided $\vt\leq\sqrt{\eta_2/c_*}$. All the lemmas above then apply.

By \eqref{fine62} and \eqref{fine39},
\[
\E(J_\al(s,x)|\mathcal F_s)\leq\exp\left(-c_2\chi_\alpha\,\vt^2\right)
\left(1+\left(\exp\big(2c_1\chi_\alpha\,\vt\big)-1\right)\kappa_0\,\vt\right).
\]
Now, note that
\[
0\;<\;\chi_\alpha\;\leq\; n^2(0).
\]
Then, by using $e^{-z}\leq 1-z/2$ on $[0,\log 2]$, we have
 \[\vt\leq\sqrt{\log 2/(c_2n^2(0))}\Longrightarrow c_2\chi_\alpha\vt^2\leq c_2n^2(0)\vt^2\leq\log 2\Longrightarrow \exp(-c_2\chi_\alpha\vt^2)\leq 1-\frac{c_2\chi_\alpha\vt^2}{2}.\] Similarly, by using
$e^z-1\leq 2z$ on $[0,\log 2]$, we have 
 \[\vt\leq\frac{\log 2}{2c_1n^2(0)}\Longrightarrow 2c_1\chi_\alpha\vt\leq\log 2\Longrightarrow \exp(2c_1\chi_\alpha\vt)-1\leq 4c_1\chi_\alpha\vt.\] 
 Thus, if we combine the two inequalities above, we obtain
\[
\E(J_\al(s,x)|\mathcal F_s)\leq\left(1-\tfrac{1}{2}c_2\chi_\alpha\vt^2\right)\big(1+4c_1\kappa_0\chi_\alpha\vt^2\big)
\leq 1-\tfrac{1}2 c_2\chi_\alpha\vt^2{2}+4c_1\kappa_0\chi_\alpha\vt^2,
\]
where we discarded the cross term which is negative.

Next, we have
\[ \vt\leq c_2/(16c_1\kappa_0)\Longrightarrow 4c_1\kappa_0\vt^2\leq c_2\vt^2/4\Longrightarrow -\tfrac{1}2 c_2\chi_\alpha\vt^2+4c_1\kappa_0\chi_\alpha\vt^2\leq-\tfrac 14 c_2\chi_\alpha\vt^2,\]
giving
\[
\E(J_\al(s,x)|\mathcal F_s)\leq 1-\tfrac 14{c_2\chi_\alpha\vt^2}\leq\exp\!\left(-\tfrac 14{c_2 \chi_\alpha\vt^2}\right).
\]
Finally, if we set
\[
\bar\vt_0:=\vt_2\wedge\sqrt{\frac{\eta_2}{c_*}}\wedge\sqrt{\frac{\log 2}{2c_2 \tilde{\nu}^2(0)}}\wedge \frac{\log 2}{2c_1 \tilde{\nu}^2(0)}\wedge \frac{c_2}{16c_1\kappa_0},\ \ \ \ 
\bar\eta_0:=c_*\,\bar\vt_0^{\,2},\ \ \ \ 
\bar\kappa:=\frac{c_2\,\bar\vt_0^{\,2}}{4},
\]
\eqref{fine15} follows.

\section{Beyond factorization and alignment}\label{sec9}

In this section we discuss possible extensions of Theorem
\ref{thm:main} beyond the two main structural assumptions used in the
proof: the {factorized form} \eqref{fine6} of the
Hamiltonian, and the \emph{double alignment} condition (A) on $\lambda$
and $\rho$. The two extensions have very different difficulty
profiles, and we treat them separately.

\subsection{Beyond double alignment}

When $\lambda$ or $\rho$ is not aligned with $H$, several structural
facts break down. In dimension~$2$, condition (T) (tangency
$\sigma\cdot\nabla H=0$) automatically gives $\sigma=\rho\,\xi$ for a
scalar amplitude $\rho$, so monotonicity of $|H|$ along the flow still
holds without aligning $\rho$. 

However
\begin{itemize}
\item[-] without alignment of $\lambda$,
$\nabla\beta_\alpha=-\al\,\nabla\lambda/(2(\lambda+\al)^2)$ is no longer
parallel to $\nabla H$, and $\xi\cdot\nabla\beta_\alpha\neq 0$. The formula
for $g_\al$ in Proposition~\ref{fine10} then picks up an extra
term $\beta_\alpha\,\rho^2\,(\xi\cdot\nabla\lambda)/\lambda^3\cdot\xi$ that
does not vanish on $\mathcal S$;
\item[-] without alignment of $\rho$, $\xi\cdot\nabla\nu\neq 0$, and a
similar extra term $\beta_\alpha\,\nu\,(\xi\cdot\nabla\nu)\,\xi$ appears.
\end{itemize}
The divergence of these extra terms involves second derivatives of
$\log\lambda$ or $\log\rho$ along $\xi$ and does not vanish on
$\mathcal S$ in general, so the sign-control argument of
Section~\ref{fine64} fails. A natural perturbative approach, 
that is writing $\lambda=l(H)+\eps\,l_1$,
$\rho=r(H)+\eps\,r_1$ for small $\eps$, could
represent a possible direction for future work.

\subsection{Beyond factorization}

The factorization $H(x_1,x_2)=h_1(x_1)\,h_2(x_2)$ is used 
throughout the paper, but the role it plays is not uniform: some uses
are essentially notational, while others are truly structural. We
find it useful to separate these cleanly because the structural obstacles are concentrated in a small number of places, and identifying them is the starting point for a future extension.

\subsubsection*{What survives without factorization}

Several pieces of the argument depend only on $H$ being a $C^3$
function on $\mathbb{T}^2$ with a suitable Morse-type structure on its
critical set, and on conditions (T) and (A). Specifically
\begin{itemize}
\item[-] Proposition~\ref{fine10} (the reduction
$g_\al=\beta_\al\nu^2\,D\xi\,\xi$). The proof uses only $\sigma=\rho\,\xi$
(from (T) in $d=2$) and that $\nabla\lambda,\nabla\rho,\nabla\nu,
\nabla\beta_\alpha$ are parallel to $\nabla H$ (from (A)). 
\item[-] Cell invariance and the monotonicity identity
\[
H(\phi_\al(t,x))\;=\;H(x)\,\exp\!\bigg(\!-\!\int_0^t
\beta_\alpha(\phi_\al(s,x))\,\nu(\phi_\al(s,x))^2\,\Lambda(\phi_\al(s,x))\,ds\bigg),
\]
where $\Lambda$ is {defined} by the algebraic identity
$\nabla H\cdot D\xi\,\xi=H\,\Lambda$ (see Lemma~\ref{fine25}).
This identity is purely algebraic in $\nabla H$ and holds for any
$C^2$ function $H$ on $\mathbb{T}^2$; one only needs to verify that
the so-defined $\Lambda$ extends continuously across $\mathcal S$. So
the dynamical step ``$|H|$ is non-increasing along the flow whenever
$\Lambda\geq 0$'' is robust, {provided} one can still establish
$\Lambda\geq 0$ globally.
\item[-] The structural decomposition $\div\,g_\al=T_1-T_2$, with
\[
T_1\;=\;\frac{\al}{l(H)+\al}\,n^2(H)\,D_H,
\qquad
T_2\;=\;H\,\Lambda\,\Theta_\alpha(H),
\]
also remains valid under (T) and (A) alone, with
$D_H:=-\det\operatorname{Hess}H$ (cf.\ Remark~\ref{fine50}). The
crucial structural feature, that is the fact that the second term carries an explicit
factor $H$ and therefore vanishes on the separatrix
$\mathcal S=\{H=0\}$, is preserved. This is the feature exploited
in the sign-control argument, and it
is {not} a consequence of factorization.

\item[-] The characteristic-flow representation and the
contraction estimate at the SDE level only use that
$|H|$ is non-increasing, that $\div\,g_\al$ has a definite sign on the
bulk and is small near the bad set, and that $\xi$ is divergence-free.
They do not use factorization.
\end{itemize}

A substantial fraction of the machinery is therefore already written in
a form that does not rely on $H$ being a product.

\subsubsection*{Where factorization is genuinely used}

The places where factorization enters in an essential way are
concentrated in three sources, all geometric in nature.

\medskip

\noindent\textbf{1. The structure of the critical set and of
$\mathcal S$.}
For $H=h_1(x_1)\,h_2(x_2)$, the separatrix decomposes as
\[
\mathcal S\;=\;(Z_1\times\mathbb{T})\cup(\mathbb{T}\times Z_2),
\]
i.e.\ as a finite union of horizontal and vertical circles, meeting
transversally at the corners $Z_1\times Z_2$ (saddles of $H$). The
cells $\mathbb{T}^2\setminus\mathcal S$ are then open
{rectangles} in the universal cover, each containing a unique
center, and each cell can be lifted to $[a_1,b_1]\times[a_2,b_2]$ with
$h_i(a_i)=h_i(b_i)=0$. This rectangular cell structure is used
\begin{itemize}
\item[-] to locate centers in $M_1\times M_2$ and corners in
$Z_1\times Z_2$ explicitly;
\item[-] to identify the midpoints
$\mathcal M=(Z_1\times M_2)\cup(M_1\times Z_2)$ as a finite, discrete
set lying on $\mathcal S$ off the corners;
\item[-] in Section \ref{fine26}, to lift the flow on a single cell to a closed
rectangle in $\mathbb{R}^2$ of side $<2\pi$, on which one-dimensional
comparison arguments can be run coordinate by coordinate.
\end{itemize}

For a general non-degenerate $H\in C^3(\mathbb{T}^2)$, $\mathcal S$ is
still a finite union of $C^3$ curves, but
these curves need not be coordinate circles, the cells need not be
rectangles, and there is no preferred set of coordinates in which the
dynamics simplifies. The discrete set $\mathcal M$ has to be
reinterpreted intrinsically. By Remark~\ref{fine50} we have
$D_H=-\det\operatorname{Hess}H$, so Lemma~\ref{fine51} gives the
intrinsic characterization
\[
\mathcal M\;=\;\big\{q\in\mathcal S\,:\,\det\operatorname{Hess}H(q)=0\big\}.
\]
Note that at such a point $\nabla H$ does not vanish (only the corners
have $\nabla H=0$, and the corners are disjoint from $\mathcal M$);
rather, $\operatorname{Hess}H$ has rank $1$ and is degenerate in the
direction tangent to $\mathcal S$. 
\medskip

\noindent\textbf{2. The sign and vanishing structure of $\Lambda$ and
$D_H$.}
Under (H1), (H2) and (H3), one obtains the global lower bound
\[
\Lambda(x)\;\geq\;2\,(h_1')^2(x_1)\,(h_2')^2(x_2)\;\geq\;0,
\qquad x\in\mathbb{T}^2,
\]
together with the precise identification $\{\Lambda=0\}=M_1\times M_2$, see Lemma~\ref{fine36}. This is the mechanism that drives
$|H|$ to $0$ along the flow at a uniform rate on $D_{\mathfrak{r}}$.
Similarly, on the separatrix one has $D_H\geq 0$ with
$\{D_H=0\}\cap\mathcal S=\mathcal M$, see Lemma~\ref{fine51}, and
at each midpoint $q\in\mathcal M$ the function $D_H$ vanishes at exact
order $2$, with the quantitative two-sided bound
\eqref{fine63} of Lemma~\ref{fine48}
\[
c_1(q)\,(x_2-p_2)^2-c_2(q)\,|x_1-q_1|
\;\leq\;D_H(x)\;\leq\;
c_2(q)\,\big((x_2-p_2)^2+|x_1-q_1|\big).
\]
These are the inputs to Lemma~\ref{fine72},  the dichotomy
$|\div g_\al|\lesssim\vt$ on $F_{\eta,\vt}$, $\div g_\al\gtrsim\vt^2$
on $G_{\eta,\vt}$,  and they are the key point of the entire
sign-control argument.

The factorized form is used here in two distinct ways.

\smallskip

\emph{Sign of $\Lambda$.} The identity
\[
\Lambda\;=\;2(h_1')^2(h_2')^2-(h_1')^2(h_2 h_2'')-(h_1 h_1'')(h_2')^2
\]
has each negative term carrying a factor $h_i\,h_i''\leq 0$ by (H3),
so dropping these terms only strengthens the bound
$\Lambda\geq 2(h_1')^2(h_2')^2\geq 0$. For a general $H$, the
intrinsic counterpart is
\[
\Lambda(x)\;=\;\frac{\nabla H(x)\cdot D\xi(x)\,\xi(x)}{H(x)},
\]
extended continuously to $\mathcal S$. There is no obvious reason for
this to be everywhere non-negative without a convexity-type assumption
on $H$ playing the role of (H3). 
\smallskip

\emph{Order-2 vanishing of $D_H$ at midpoints.} The estimate \eqref{fine63} is asymmetric in $x_1$ and
$x_2$ at $q=(q_1,p_2)\in Z_1\times M_2$, as the dominant positive term is
quadratic in $x_2-p_2$, and the negative correction
is linear in $|x_1-q_1|$. 
For non-factorized $H$, there is no global splitting of $\mathbb{T}^2$
into ``longitudinal'' and ``transverse'' coordinates; the analogous
estimate must be formulated intrinsically near each $q\in\mathcal M$
in the form
\[
D_H(x)\;\geq\;c_q\,\dist(x,\mathcal M)^2-C_q\,\dist(x,\mathcal S),
\]
provided $D_H$ vanishes to exactly order $2$ along $\mathcal S$ at $q$
and to first order transverse to $\mathcal S$. This is in principle
still achievable, since $\mathcal M$
is the set where $\det\operatorname{Hess}H$ vanishes on $\mathcal S$,
and at such points one expects $D_H$ to behave like the squared
distance along $\mathcal S$. However, it is no longer a one-line
consequence of Taylor expansion in product coordinates and would
require a non-trivial geometric argument.

\medskip

\noindent\textbf{3. The componentwise SDE in product coordinates.}
In Section~\ref{fine26}, the localization
argument expresses the noise component $\sigma_\al(x)=\nu(x)\,\xi(x)$
in coordinates as
\[
\sigma_{\al,1}(x)\;=\;-\nu(x)\,h_1(x_1)\,h_2'(x_2),
\qquad
\sigma_{\al,2}(x)\;=\;\nu(x)\,h_1'(x_1)\,h_2(x_2),
\]
and exploits that each component is a product of a function of $x_1$
and a function of $x_2$. This is what makes the cell-arc analysis
(one-dimensional comparison on $[a_1,b_1]$ and $[a_2,b_2]$ separately)
work. For a general $H$, $\xi$ is a single $C^2$ vector field on
$\mathbb{T}^2$ with no such product structure, and the localization
lemma would have to be reformulated.

\subsubsection*{A strategy for a future extension}

The above analysis suggests that an extension to non-factorized $H$
would proceed by replacing the explicit product structure with the
following intrinsic inputs, treated as assumptions on $H$:
\begin{itemize}
\item[-] $H\in C^3(\mathbb{T}^2)$ with non-degenerate critical
points (Morse), so that the centers (local extrema), saddles
(replacing the corners), and the midpoints
$\mathcal M=\{q\in\mathcal S\,:\,\det\operatorname{Hess}H(q)=0\}$ form
three pairwise disjoint finite sets;
\item[-] $\Lambda\geq 0$ on $\mathbb{T}^2$ with
$\{\Lambda=0\}\subset\{\nabla H=0\}$ (a global ``convexity along the
flow'' assumption replacing (H3));
\item[-] $D_H\geq 0$ on $\mathcal S$ with
$\{D_H=0\}\cap\mathcal S=\mathcal M$, vanishing at exact order $2$ at
each midpoint in the longitudinal direction (replacing
Lemmas~\ref{fine51} and \ref{fine48}).
\end{itemize}
Under such assumptions, Sections~\ref{fine64} to \ref{fine56}
should go through with only notational changes, since they use only
the bounds on $\Lambda$, $D_H$, and the geometric-measure properties
of $\mathcal S$, $\mathcal M$, and the cells. Section \ref{fine26} is the place where the most rewriting would be needed

\section*{Acknowledgements} 
This material is based upon work supported by the National Science
Foundation under Grant No. DMS-2424139, while the  authors were in
residence at the Simons Laufer Mathematical Sciences Institute in
Berkeley, California, during the Fall 2025 semester.

\end{document}